\documentclass[11pt]{amsart} \textwidth=14.5cm \oddsidemargin=1cm
\usepackage{amsmath,wasysym}
\usepackage{amsxtra}
\usepackage{amscd}
\usepackage{amsthm}
\usepackage{amsfonts}
\usepackage{amssymb}
\usepackage{arydshln}
\usepackage{eucal}
\usepackage{graphics, color}
\usepackage{hyperref,mathrsfs}
\usepackage[usenames,dvipsnames]{xcolor}
\usepackage{tikz}
\usepackage{stmaryrd}
\usepackage[all]{xy}
\usepackage{hyperref}
\usepackage{color}
\usepackage{bbm} % \mathbbm
\allowdisplaybreaks

\hypersetup{colorlinks,linkcolor=blue, urlcolor=blue,
citecolor=blue}

\newtheorem{thm}{Theorem} [section]
\newtheorem{lem}[thm]{Lemma}
\newtheorem{cor}[thm]{Corollary}
\newtheorem{prop}[thm]{Proposition}

\newtheorem{rmk}[thm]{Remark}

\theoremstyle{definition}
\newtheorem{definition}[thm]{Definition}

\DeclareMathOperator{\Hom}{Hom}

\newcommand{\binnk}{\mathscr{F}^n_k}
\newcommand{\C}{\mathbb{C}}
\newcommand{\cell}{\Delta}

\newcommand{\mH}{\mathscr{H}_{n,R}(q)}
\newcommand{\mHK}{\mathscr{H}_{n,\mathbb{K}}(q)}

\newcommand{\Hta}{\mathcal{H}'_{n-k}}

\newcommand{\ks}{\mathfrak{s}}
\newcommand{\kt}{\mathfrak{t}}
\newcommand{\ku}{\mathfrak{u}}
\newcommand{\kv}{\mathfrak{v}}

\newcommand{\Mat}{\operatorname{Mat}}

\newcommand{\Sn}{{\mathfrak{S}}}
\newcommand{\Std}{\mathrm{Std}}

\newcommand{\wtT}{\widetilde T}

\newcommand{\nc}{\newcommand}
\nc{\redtext}[1]{\textcolor{red}{#1}}
\nc{\red}[1]{\redtext{ #1}}
\nc{\browntext}[1]{\textcolor{brown}{#1}}
\nc{\greentext}[1]{\textcolor{green}{#1}}
\nc{\bluetext}[1]{\textcolor{blue}{#1}}
\nc{\brown}[1]{\browntext{ #1}}
\nc{\green}[1]{\greentext{ #1}}
\nc{\blue}[1]{\bluetext{ #1}}

\numberwithin{equation}{section}

\title[Cellular Bases]
{Cellular Bases of Mirabolic Hecke Algebras}

\author[Jian Chen]{Jian Chen}
\address{School of Mathematical Sciences,
	University of Science and Technology of China,
	Hefei 230026, China}
\email{JianChen96@outlook.com (Chen)}

\author[Yan Liu]{Yan Liu}
\address{School of Mathematics and Statistics, Beijing Institute of Technology, Beijing 100081, China}
\email{3120251485@bit.edu.cn(Liu)}

\author[Jinkui Wan]{Jinkui Wan}
\address{School of Mathematical Sciences,
	Shenzhen University, Shenzhen, 518060, China}
\email{wjk302@hotmail.com (Wan)}

\author[Xirui Yu]{Xirui Yu}
\address{School of mathematical Sciences,
    East China Normal University, Shanghai 200241, China}
\email{52280155007@stu.ecnu.edu.cn (Yu)}

\begin{document}

\begin{abstract}
{\color{black}
We prove that the mirabolic Hecke algebra $\mathscr{H}_{n,R}(q)$ over a commutative ring $R$ is cellular by constructing explicit cellular bases. When $R=\mathbb{K}$ is a field, we also establish a Morita equivalence between $\mathscr{H}_{n,\mathbb{K}}(q)$ and a direct sum of Hecke algebras associated with symmetric groups. As applications, we classify the irreducible $\mHK$-modules and obtain a semisimplicity criterion. Finally, when $\mathbb{K}=\mathbb{C}(q)$, we show that the cell modules arising from one of our cellular bases coincide with Siegel's irreducible $\mathscr{H}_{n,\mathbb{K}}(q)$-modules.
}
\end{abstract}

\maketitle
\setcounter{tocdepth}{1}
\tableofcontents
%

%%%%%%%%%%%
%%%%%%%%%%%
%%%%%=================================================================================================

\section{Introduction}
%%%%%==================================================================================================
{\color{black}

The study of the mirabolic Hecke algebra $\mH$ was initiated by Solomon \cite{So72}, and its irreducible representations were subsequently described by Siegel \cite{Si93}. The algebra $\mH$ shares many features with the Hecke algebra $\mathcal{H}_{n,R}(q)$. For example, when $R=\mathbb{C}$ and $q$ is a prime power, Rosso \cite{Ro14} identified $\mH$ with the convolution algebra of $GL_n(q)$-equivariant complex-valued functions on the triple variety
$
GL_n(q)/\mathsf{B} \times GL_n(q)/\mathsf{B} \times \mathbb{F}_q^n,
$
where $\mathsf{B}$ is the Borel subgroup of upper triangular matrices in $GL_n(q)$. This generalizes the classical realization of $\mathcal{H}_n(q)$ as the convolution algebra of $GL_n(q)$-equivariant complex-valued functions on the double flag variety
$
GL_n(q)/\mathsf{B} \times GL_n(q)/\mathsf{B}.
$
Subsequent developments have further clarified the connections between mirabolic Hecke algebras and quantum groups. In particular, Rosso's work on mirabolic quantum groups \cite{Ro18}, together with the subsequent work of Fan, Zhang, and Ma \cite{FZM25}, established mirabolic Schur--Weyl duality, relating the representation theory of mirabolic Hecke algebras to that of their quantum-group counterparts. More recently, the third author \cite{Wan26} established a Schur--Weyl duality between the mirabolic Hecke algebra and the usual quantum group $U_q(\mathfrak{gl}_r)$, providing a new connection between mirabolic representation theory and the classical framework of quantum groups. Her work also gives a new presentation and basis for the mirabolic Hecke algebra, as well as Frobenius character formulas and a recursive Murnaghan--Nakayama rule for its irreducible characters. These developments reveal a rich interplay among the geometric, algebraic, and representation-theoretic aspects of mirabolic Hecke algebras and motivate further study of their structure.

The main goal of this article is to prove that the mirabolic Hecke algebra $\mH$ is cellular, as conjectured in \cite{Ro14}. Rather than proving directly that the elements proposed in \cite[Conjecture 7.5]{Ro14} form a cellular basis when $R=\mathbb{C}(q)$, we construct a new family of cellular bases. More precisely, using the idempotents $P_1,P_2,\ldots,P_n$ from the presentation of $\mH$ in \cite{Wan26}, we first construct a filtration
\[
0 = J_{n+1} \subset J_n \subset J_{n-1} \subset \cdots \subset J_1 \subset J_0 = \mH
\]
by two-sided ideals. By modifying the $R$-basis of $\mH$ introduced in \cite{Wan26}, we obtain a basis for each subquotient $J_k/J_{k+1}$, where $0\leq k\leq n$, and derive an explicit multiplication formula in terms of multiplication in the Hecke algebra $\mathcal{H}'_{n-k,R}(q)$ associated with the symmetric group; see Proposition \ref{prop:layer-mult-basis}. This allows us to prove that a choice of cellular basis for each $\mathcal{H}'_{n-k,R}(q)$, with $0\leq k\leq n$, gives rise to a cellular basis for $\mH$ in the sense of \cite{GL96}.

 When $R=\mathbb{K}$ is a field, we  show that $\mHK$ is isomorphic to the direct sum of the subquotients $J_k/J_{k+1}$ for $0\leq k\leq n$ and  further identify $J_k/J_{k+1}$ with
$
V_k\otimes_R V_k\otimes_R \mathcal{H}'_{n-k,R}(q),
$
where $V_k$ is a free $R$-module of rank $\binom{n}{k}$. This yields an isomorphism between $\mHK$ and a direct sum of matrix algebras over the Hecke algebras $\mathcal{H}'_{n-k,\mathbb{K}}(q)$, that is, 
\[
\mHK\cong\oplus_{k=0}^n\text{Mat}_{\binom{n}{k}}(\mathcal{H}'_{n-k,\mathbb{K}}(q)). 
\]
Consequently, $\mHK$ is Morita equivalent to the direct sum of $\mathcal{H}'_{n-k,\mathbb{K}}(q)$ with $0\leq k\leq n$. 
This equivalence allows us to deduce the representation theory of $\mHK$ from that of the Hecke algebras $\mathcal{H}_{r,\mathbb{K}}(q)$ for $0\leq r\leq n$. In particular, it yields a classification of irreducible modules and a semisimplicity criterion, and determines the decomposition numbers and block decomposition of $\mHK$.

Our calculations parallel those for the $q$-rook monoid algebra in \cite{Pa06} in several respects. In fact, we establish an explicit connection between $\mH$ and the $q$-rook monoid algebras, and our construction also yields cellular bases for the latter; see Remark \ref{rmk:rook-correction}.

Using the cellular basis of $\mH$ arising from the Murphy bases of the Hecke algebras, we construct the corresponding cell modules and compute their canonical bilinear forms. We then apply the general theory of cellular algebras \cite{GL96} to recover the classification of irreducible $\mHK$-modules obtained via the Morita equivalence above. When $\mathbb{K}=\mathbb{C}(q)$, we further show that these cell modules coincide with the irreducible modules constructed in \cite{Si93,Ro14}.

Our construction also raises several questions. One is how our cellular bases relate to the elements proposed in \cite[Conjecture 7.5]{Ro14}. Another is whether a cellular basis of $\mH$ arising from Kazhdan--Lusztig-type canonical bases of the Hecke algebras $\mathcal{H}'_{n-k,R}(q)$ yields a Kazhdan--Lusztig-type canonical basis of $\mH$, and how such a basis relates to the basis of the mirabolic bimodule constructed by Travkin \cite{Tr09}. We will address these questions in a forthcoming paper \cite{LW26}.

The paper is organized as follows. In Section \ref{sec: basics}, we review the basic properties of the mirabolic Hecke algebra $\mH$ and introduce a new basis used throughout the subsequent sections. In Section \ref{sec:cellular}, we establish key multiplication formulas, construct cellular bases for $\mH$, and prove the isomorphism with a direct sum of matrix algebras over Hecke algebras associated with symmetric groups. In Section \ref{sec:cell}, we describe the cell modules and show that, when $\mathbb{K}=\mathbb{C}(q)$, they coincide with Siegel's irreducible $\mHK$-modules.

Throughout the article, for each non-negative integer $n\geq 0$, and we write
$
[n]=\{1,2,\dots,n\}$ and $
[0]=\varnothing.
$
For $0\leq k\leq n$, set
\[
\mathscr{F}^n_k=\{A\subseteq [n]\mid |A|=k\},
\]
where $|A|$ denotes the cardinality of $A$. Thus, $\mathscr{F}^n_k$ is the set of all $k$-element subsets of $[n]$, and
$
|\mathscr{F}^n_k|=\binom{n}{k}.
$
Let $\Sn_n$ denote the symmetric group on $[n]$, with simple reflections $s_i=(i,i+1)$ for $1\leq i\leq n-1$, and let $\ell$ denote its length function.

%We denote $\Zq=\Z[q,q^{-1}]$.
%%%%%%%%%%%
\subsection*{Acknowledgement}
The first author is supported by NSFC-12671043. The second and third authors are supported by NSFC-12671038. 
}
%%%%%%%%%
%%%%%%%%%
%%%%%==================================================================================================
{\color{black}
\section{Mirabolic Hecke algebra $\mH$ and its bases}\label{sec: basics}
%%%%%================================
%%%%%==================================================================================================
In this section, we first recall the definition of the mirabolic Hecke algebra in \cite{So72} as well as the new presentation in \cite{Wan26}. Then we derive a new basis from the one given in \cite{Wan26}.
%%%%%%%
\subsection{Mirabolic Hecke algebra}

Suppose $R$ is a commutative ring. Let $q\in R$ be an invertible element.
The following definition of the mirabolic Hecke algebra is due to \cite{So72}.
%%%%%%%%
\begin{definition}\cite{So72}\label{def:mirabolic}
 The mirabolic Hecke algebra $\mH$ is defined to be the $R$-algebra generated by $T_0,T_1,\dots,T_{n-1}$, subject to the following relations:
 $$
\begin{aligned}
T_0^2&=(q-2)T_0+(q-1); \quad T_i^2=(q-1)T_i+q, \quad 1\le i< n;\\
T_iT_j&=T_jT_i, \quad  0\le i,j<n,~|i-j|>1;\quad T_iT_{i+1}T_i=T_{i+1}T_iT_{i+1}, \quad 1\le i\le n-2;\\
T_0T_1T_0T_1&=(q-1)(T_1T_0T_1+T_1T_0)-T_0T_1T_0;\\
 T_1T_0T_1T_0&=(q-1)(T_1T_0T_1+T_0T_1)-T_0T_1T_0.
\end{aligned}
$$
\end{definition}

It follows from \cite{So72} that the subalgebra of $\mH$ generated by $T_1,\ldots,T_{n-1}$ is isomorphic to the finite Hecke algebra $\mathcal{H}_{n,R}(q)$ associated to symmetric group $\mathfrak{S}_n$. 
 In order to establish a duality between the quantum groups $U_q(\mathfrak{gl}_r)$ and $\mH$ over $R=\mathbb{C}(q)$ which generalizes the Jimbo-Schur-Weyl duality for $\mathcal{H}_{n,\mathbb{C}(q)}(q)$,   a new representation of $\mH$ is provided in \cite[Theorem~2.8]{Wan26}. 
 First we write 
 \begin{equation}\label{eq:Pi-recursive-1}
 P_1=1-q^{-1}(T_0+1), \quad P_i=-P_{i-1}T_{i-1}^{-1}P_{i-1}\quad (2\leq i\leq n). 
 \end{equation}
\begin{prop}\cite[Theorem 2.8]{Wan26}\label{prop:mirabolic}
The mirabolic Hecke algebra $\mH$ is generated by $T_1,\dots,T_{n-1}, P_1,\dots,P_n$ over $R$,
subject to the following relations:
\begin{align}
T_i^2&=(q-1)T_i+q,\qquad (1\le i\le n-1)\label{eq:quad}\\
T_iT_{i+1}T_i&=T_{i+1}T_iT_{i+1},\qquad T_iT_j=T_jT_i\quad (|i-j|>1),\label{eq:braid}\\
P_i^2&=P_i,\qquad (1\le i\le n),\label{eq:Pidem}\\
P_iP_j&=P_jP_i=P_i\qquad (1\le j<i\le n),\label{eq:Pnest}\\
P_iT_j&=T_jP_i\qquad (1\le i<j\le n-1),\label{eq:Pcomm}\\
P_iT_j&=T_jP_i=-P_i\qquad (1\le j<i\le n),\label{eq:Pminus}\\
P_i&=-P_{i-1}T_{i-1}^{-1}P_{i-1}\qquad (2\le i\le n).\label{eq:Prec}
\end{align}
\end{prop}
Using the convention, we set $P_0=1$. Obviously, by \eqref{eq:quad} and \eqref{eq:Pi-recursive-1} we have
\begin{equation}\label{eq:add-mirabolic-recursion}
    P_iT_iP_i=(q-1)P_i-qP_{i+1}.
\end{equation}

For $0\le d\le n$, denote by $\Sn_d'$ the Young subgroup of $\Sn_n$ corresponding to the composition $\lambda=(1^{n-d},d)=(1,\ldots,1,d)$. It is obvious that $\Sn_d'\cong \Sn_d$.

%%%%%%%%%%%%%%
\subsection{The bases of mirabolic Hecke algebra}
For each 
$A=\{a_1<a_2<\cdots<a_k\}\in \mathscr{F}^n_k$, set 
\begin{equation}\label{eq:wA-reduced}
w_A=(s_{a_1-1}\cdots s_1)(s_{a_2-1}\cdots s_2)\cdots (s_{a_k-1}\cdots s_k).
\end{equation}
Clearly, the right hand side of \eqref{eq:wA-reduced} is a reduced expression and hence $\ell(w_A)=\sum_{r=1}^k (a_r-r)$. 
Actually, if writing $[n]\setminus A=\{c_1<c_2<\cdots<c_{n-k}\}$, then the permutation $w_A$ is actually the $(k,n-k)$-shuffle corresponding to $A$, that is it satisfies 
\begin{equation*}
w_A=(a_1,a_2,\dots,a_k,c_1,c_2,\dots,c_{n-k})\in \Sn_n \quad (\text{in one-line notation}). 
\end{equation*}
 Equivalently,  $w_A$ is the unique shortest permutation such that 
$ w(\{1,\dots,k\})=A. $

Setting $T_{j,r}=T_{j-1}T_{j-2}\cdots T_r$ for $1\le r<j\le n$ and $T_{r,r}=1$, then for each subset $A=\{a_1<\cdots<a_k\}\subseteq [n]$, define
$$T_A=T_{a_1,1}T_{a_2,2}\cdots T_{a_k,k}.$$
Note that $T_{[n]}=1$. By~\eqref{eq:wA-reduced}, we have $$T_A=T_{w_A}.$$
For each $0 \leq k \leq n$, set
$$\Gamma_k = \{(A,B,w) \mid A,B \in \mathscr{F}^n_k, w \in \mathfrak{S}_{n-k}'\}$$
and let $\Gamma = \cup_{k=0}^n \Gamma_k$. For each $(A,B,w) \in \Gamma_k$, define
$$T_{(A,B,w)} := T_A P_k T_w T_B^{-1}. $$

\begin{lem}\cite[Proposition 2.6]{Wan26}\label{lem:basis-Wan}
    The algebra $\mH$ is a free $R$-module and the set $\{ T_{(A,B,w)} \mid (A,B,w) \in \Gamma \}$ forms a $R$-basis for $\mH$.
\end{lem}

In the following we shall introduce another basis of $\mH$ which is useful for showing the cellularity of $\mH$ in subsequent sections. 
It is straightforward to check that the following map
\begin{align*}
  \sigma: \mH &\to \mH\\
T_i &\mapsto T_i,\quad 
P_j \mapsto P_j ,\quad (1\leq i\leq n-1, 1\leq j\leq n)
\end{align*}
is an anti-involution on $\mH$.  Set
\begin{equation}\label{eq:tilde-basis}
    \wtT_{(A,B,w)} := \sigma(T^{-1}_A) P_k T_w T_B^{-1}.
\end{equation}
for $(A,B,w) \in \Gamma_k$ and $k =0,1,\cdots,n.$ 
Then clearly
\begin{equation}\label{eq:imath-T}
\sigma(\wtT_{(A,B,w)} )=\sigma(T^{-1}_B) P_k T_{w^{-1}} T_A^{-1}=\wtT_{(B,A,w^{-1})}
\end{equation}
\begin{prop}\label{pro:basis}
For $0\leqslant k\leqslant n$, the set $\{\widetilde{T}_{(A,B,w)} | (A,B,w) \in \Gamma\}$ is a $R$-basis of $\mH$.
\end{prop}

\begin{proof} 
 The method we will apply to prove the proposition is parallel to the one used in the proof of \cite[Lemma 2.5]{Wan26} or \cite[Theorem 2.1]{Ha04}. 
 Set $\mathcal{B}=\{\widetilde{T}_{(A,B,w)}\mid (A,B,w)\in\Gamma\}$ and let $\mathscr{H}'_{n,R}(q)$ be the $R$-submodule of $\mH$ spanned by $\mathcal{B}$.  Clearly by \eqref{eq:imath-T} we have 
 \begin{equation}\label{eq:imath-H}
 \sigma\big(\mathscr{H}'_{n,R}(q) \big)=\mathscr{H}'_{n,R}(q). 
 \end{equation}
We claim that $\mathscr{H}'_{n,R}(q)=\mH$. Obviously it suffices to show $\mathscr{H}_{n,R}(q) \subseteq\mathscr{H}'_{n,R}(q)$.
Since $\mathscr{H}_{n,R}(q)$ is generated by $P_1, T_1, \dots, T_{n-1}$, each element of $\mathscr{H}_{n,R}(q)$ can be expressed as a linear combination of monomials in these generators.  
Note $T_1, \dots, T_{n-1}, P_1 \in \mathscr{H}'_{n,R}(q)$, because $T_i = T_{(\varnothing,\varnothing,s_i)}$ for $1 \le i \le n-1$, and $P_1 = T_{(\{1\},\{1\},\text{id})}$.  
Then by induction on the number of factors in monomials in $P_1, T_1, \dots, T_{n-1}$, it suffices to show that
\[
\mathscr{H}'_{n,R}(q) T^{-1}_i \subseteq \mathscr{H}'_{n,R}(q) \quad (1 \le i \le n-1),\qquad 
\mathscr{H}'_{n,R}(q) P_1 \subseteq \mathscr{H}'_{n,R}(q),
\]
in order to prove $\mathscr{H}_{n,R}(q) \subseteq\mathscr{H}'_{n,R}(q)$. 

Using a strategy parallel to the proof of \cite[Theorem 2.1]{Ha04}, one can verify $\mathscr{H}'_{n,R}(q) T_i \subseteq \mathscr{H}'_{n,R}(q)$ for $1 \le i \le n-1$.  
We shall include the details in the following for completeness though the analogous calculations  are omitted in the proof of \cite[Lemma 2.5]{Wan26}. 

Take $(A,B,w)\in\Gamma$ and $1\leq i\leq n-1$. We first take care of the multiplication $\wtT_{(A,B,w)}= \sigma(T^{-1}_A) P_k T_w T_B^{-1}T_i^{-1}$ by $T_i^{-1}$. 
There are four possibilities depending on whether $i$ or $i+1$ belongs to $B$.

    If $i,i+1\in B$, one can write $B=\{b_1<\cdots<b_j<b_{j+1}<\cdots<b_k\}$ with $b_j=i, b_{j+1}=i+1$, \textcolor{black}{$j\leq i$ and $j< k$} .  Accordingly  we can write $T_B^{-1}=X\,T_{i+1,j+1}^{-1}T_{i,j}^{-1}\,Y$, where $X\in \mathcal{H}_{n,R}(q)$ commutes with $T_j^{-1}$ and $Y\in \mathcal{H}_{n,R}(q)$ commutes with $T_i^{-1}$. Consecutively applying \eqref{eq:braid} and \eqref{eq:Pidem} leads to 
    \[
    (T_{i+1,j+1}^{-1}T_{i,j}^{-1})T_i^{-1}=T_j^{-1}(T_{i+1,j+1}^{-1}T_{i,j}^{-1}), 
    \]
    which leads to 
$$
    \begin{aligned}
   \wtT_{(A,B,w)}T_i^{-1}&=  \sigma(T^{-1}_A) P_k T_w T_B^{-1}T_i^{-1}\\
   &=\sigma(T^{-1}_A) T_wP_k  X\,T_j^{-1}T_{i+1,j+1}^{-1}T_{i,j}^{-1}\,Y\\
   &=\sigma(T^{-1}_A) T_w P_kT_j^{-1}X\,T_{i+1,j+1}^{-1}T_{i,j}^{-1}\,Y\\
  & =\sigma(T^{-1}_A) T_w(-P_k) X\,T_{i+1,j+1}^{-1}T_{i,j}^{-1}\,Y=-\sigma(T^{-1}_A) P_kT_wT_B^{-1}
   \end{aligned}
  $$
   since $j<k$. Thus this case is proved. 

    If $i,i+1\notin B=\{b_1<\cdots<b_k\}$, then there exists $0\leq l\leq k$ such that $b_l<i, b_{l+1}>i+1$ \textcolor{black}{or all $b_l<i$}, where we use the convention $b_0=0$. 
    \textcolor{black}{In first case}, write $T^{-1}_B=XY$ where $X=T_{b_k,k}^{-1}\cdots T_{b_{l+1},l+1}^{-1}$ and $Y=T_{b_{l},l}^{-1}\cdots T_{b_1,1}^{-1}$. Clearly $YT_i^{-1}=T_i^{-1}Y$. Meanwhile observe that by \eqref{eq:braid}  we always have 
    $$
    T^{-1}_{c,a}T_b^{-1}=T_{b+1}^{-1}T^{-1}_{c,a}
    $$
    for any triple $a,b,c$ with $a\leq b< \color{black}{c-1}$. This together with $l\leq b_l\leq i+1,  b_{l+1}>i+1$ leads to 
    $$
    \begin{aligned}
    XT_i^{-1}=T_{b_k,k}^{-1}\cdots T_{b_{l+1},l+1}^{-1}T_i^{-1}&=T_{b_k,k}^{-1}\cdots T_{b_{l+2},l+2}^{-1} T^{-1}_{i+1}T_{b_{l+1},l+1}^{-1}=\cdots=T^{-1}_{i+k-l} X. 
    \end{aligned}
    $$
    Therefore 
    $$
    \begin{aligned}
     \wtT_{(A,B,w)}T_i^{-1}&= \sigma(T^{-1}_A) P_k T_w XY T_i^{-1}=\sigma(T^{-1}_A) P_k T_w T^{-1}_{i+k-l} XY\\
    & =\sigma(T^{-1}_A) P_k T_w T^{-1}_{i+k-l}T_B^{-1}. 
     \end{aligned}
    $$
    This case is proved as $T_w T^{-1}_{i+k-l}$ is a linear combination of $T_w$ and $T_{ws_{i+k-l}}$ and $ws_{i+k-l}\in \mathfrak{S}'_{n-k}$ due to the fact $i>l$.   
    
     \textcolor{black}{And if all $b_l<i$ ,we have $k<i$. So $\wtT_{(A,B,w)}T_i^{-1}=\sigma(T^{-1}_A) P_k T_w T^{-1}_{i}T_B^{-1}$ where $k<i$. $T_w T^{-1}_{i}$ is a linear combination of $T_w$ and $T_{ws_{i}}$ and $ws_{i}\in \mathfrak{S}'_{n-k}$.}
    
    If $i\in B$, $i+1\notin B$, write $B=\{b_1<\cdots<b_{l-1}<b_l=i<b_{l+1}<\cdots<b_k\}$ for some $1\leq l\leq k$. Accordingly  write $T_B=X T_{i,l}^{-1}Y$ where $Y=T_{b_{l-1},l-1}^{-1}\cdots T_{b_1,1}^{-1}$ 
     and $X=T_{b_k,k}^{-1}\cdots T_{b_{l+1},l+1}^{-1}$. Then
    \[
    \begin{aligned}
    \wtT_{(A,B,w)}T_i^{-1}&=  \sigma(T^{-1}_A) P_k T_wT_BT_i^{-1}=\sigma(T^{-1}_A) P_k T_wX T_{i,l}^{-1}T_i^{-1}Y\\
    &=\sigma(T^{-1}_A) P_k T_wX T_{i+1,l}^{-1}Y=\sigma(T^{-1}_A) P_k T_wT_{B'}^{-1},
    \end{aligned}
    \]
    where $B'=\{b_1<\cdots<b_{l-1}<i+1<b_{l+1}<\cdots<b_k\}$. Thus this case is proved. 
    
    If $i\notin B$, $i+1\in B$, likewise write $T_B=X T_{i+1,l}^{-1}Y$. By \eqref{eq:quad} we have 
    \[
    T_BT_i^{-1}=X T_{i+1,l}^{-1}T_i^{-1}Y
    =(q^{-1}-1)X T_{i+1,l}^{-1}Y+q^{-1}X T_{i,l}^{-1}Y
    =(q^{-1}-1)T_B^{-1}+q^{-1}T_{B'}^{-1},
    \]
    where $B'$ is obtained from $B$ with $i+1$ replaced by $i$. Hence this case is proved. Putting together, we obtain
    \begin{equation}\label{eq:tilde-Ti}
    \mathscr{H}'_{n,R}(q) T^{-1}_i \subseteq \mathscr{H}'_{n,R}(q),  \quad 1 \le i \le n-1. 
    \end{equation}
This together with \eqref{eq:imath-T} and \eqref{eq:imath-H} gives rise to 
 \begin{equation}\label{eq:tilde-Ti-left}
  T^{-1}_i \mathscr{H}'_{n,R}(q)  \subseteq \mathscr{H}'_{n,R}(q),  \quad 1 \le i \le n-1. 
    \end{equation}

    \medskip
    \noindent Next take $(A,B,w)\in\Gamma$ with $B=\{b_1<b_2<\cdots<b_k\}$ and we shall consider the multiplication by $\wtT_{(A,B,w)}= \sigma(T^{-1}_A) P_k T_w T_B^{-1}P_1$ by $P_1$ in three cases.

    If $1\in B$, then both $T_w$ and $T_B^{-1}$ commute with $P_1$ by (\ref{eq:Pcomm}) and $P_kP_1=P_k$. So
    \[
   \wtT_{(A,B,w)} P_1= \sigma(T^{-1}_A) P_kP_1 T_w T_B^{-1}= \wtT_{(A,B,w)} 
    \]

    If $1\notin B$ and $B\neq\varnothing$, then $b_j>j$ for each $1\leq j\leq k$. Observe that by \eqref{eq:Prec} we always have 
    \begin{equation}\label{eq:PTP}
    P_iT_{b,i}^{-1}P_i=-P_{i+1}T_{b,i+1}^{-1},
    \end{equation} 
    for any $1\leq i<b\leq n$. 
   Then repeatedly applying \eqref{eq:PTP} one can get 
    \[
    \begin{aligned}
    P_kT_B^{-1}P_1
    &=P_kP_1T_B^{-1}P_1
    =P_k(T_{b_k,k}^{-1}\cdots T_{b_2,2}^{-1})(P_1T_{b_1,1}^{-1}P_1)\\
    &=-P_k(T_{b_k,k}^{-1}\cdots T_{b_2,2}^{-1})P_2T_{b_1,2}^{-1}\\
    &=-P_kP_2(T_{b_k,k}^{-1}\cdots T_{b_2,2}^{-1})P_2T_{b_1,2}^{-1}\\
    &=\cdots\\
    &=(-1)^{k}P_{k+1}(T_{b_k,k+1}^{-1}\cdots T_{b_2,3}^{-1}T_{b_1,2}^{-1})
    =(-1)^{k}P_{k+1}T_{B'}^{-1},
    \end{aligned}
    \]
    with $B'=\{1<b_1<\dots<b_k\}$.
    So
    \begin{equation}\label{eq:wtT-P1}
    \sigma(T_A^{-1})P_kT_wT_B^{-1}P_1= \sigma(T_A^{-1})T_wP_kT_B^{-1}P_1
    =(-1)^k\sigma(T_A^{-1})T_wP_{k+1}T_{B'}^{-1}.
    \end{equation}
    Meanwhile by $P_{k+1}T_{B'}^{-1}=\wtT_{(C, B',\text{id})}$ $C=\{1,2,\ldots,k+1\}$. Thus $P_{k+1}T_{B'}^{-1}\in\mathscr{H}'_{n,R}(q)$.  Then by \eqref{eq:tilde-Ti-left} and \eqref{eq:wtT-P1} one can deduce that 
    $\wtT_{(A,B,w)}P_1\in\mathscr{H}'_{n,R}(q)$. 
    
 The remaining situation is $B=\varnothing$. Correspondingly,  we need to prove that $T_wP_1$ is in the $R$-span of $\mathcal{B}$. This is straightforward due to \eqref{eq:tilde-Ti-left}.  
 
 Putting together, we obtain that $\mH=\mathscr{H}'_{n,R}(q)$, that is, $\mH$ is spanned by the set $\{\wtT_{(A,B,w)}\mid (A,B,w)\in\Gamma\}$. 
The fact $|\mathcal{B}|=|\Gamma|$ and Lemma \ref{lem:basis-Wan} imply that $\mathcal{B}$ is also a basis and thus the proposition is proved.   
\end{proof}
%%%%%%%%%%%%%

Let $J_k$ be the two-sided ideal of $\mH$ generated by $P_k$ and we apply the convention $P_0=1, P_{n+1}=0 $. Then, we have that 
$$0 = J_{n+1} \subset J_n \subset J_{n-1} \subset \cdots \subset J_1 \subset J_0 = \mH$$ 
is a chain of 2-sided ideal. Moreover, we have,
\begin{thm}\label{thm:Ji-basis}
The set $\{\widetilde{T}_{(A,B,w)} \mid (A,B,w) \in \cup_{k=i}^n \Gamma_k\}$ is a basis of the two-sided ideal $J_i$ for each $0\leq i\leq n+1$.
\end{thm}
\begin{proof}
Fix $0\leq i\leq n+1$. Observe that the set $\{ \widetilde{T}_{(A,B,w)} \mid (A, B, w) \in \bigcup_{k=i}^{n} \Gamma_k \}$ is linearly independent since it is a subset of the basis $\{ \widetilde{T}_{(A,B,w)} \mid (A, B, w) \in \Gamma \}$ of $\mH$ given in Proposition \ref{pro:basis}. 
 Note that for arbitrary $k\geq i$, by \eqref{eq:Pnest} we have $P_k=P_kP_i\in J_i$ and hence  $\widetilde{T}_{(A,B,w)}=\sigma(T_A^{-1})T_wP_kT_B^{-1} \in J_i$ for all $(A, B, w) \in \bigcup_{k=i}^{n} \Gamma_k$.   It remains to show that 
$J_i$ is spanned by $\{ \widetilde{T}_{(A,B,w)} \mid (A, B, w) \in \bigcup_{k=i}^{n} \Gamma_k \}$.  Let $J_i'$ be the $R$-span of $\{ \widetilde{T}_{(A,B,w)} \mid (A, B, w) \in \bigcup_{k=i}^{n} \Gamma_k \}$. So $J_i'\subseteq J_i$. Then it suffices to prove 
$J_i\subseteq J_i'$. On the other hand, as $J_i$ is the two-sided ideal of $\mH$ generated by $P_i$ and moreover $P_i = \widetilde{T}_{(A,A,\text{id})}\in J_i'$ with $A=\{1<2<\cdots<i\}$. Thus it suffices to prove that $J_i'$ is a two-sided ideal in $\mH$. 
By the proof in Proposition \ref{pro:basis}, we can easily check that $\widetilde{T}_{(A,B,w)}T_l, T_l\widetilde{T}_{(A,B,w)}, \widetilde{T}_{(A,B,w)}P_1, P_1\widetilde{T}_{(A,B,w)}\in \color{black}{J'_i}$ for $1\leq l\leq n-1$ and $(A,B,w)\in \Gamma_k$ with $k\geq i$. 
This together with the fact that $\mH$ is generated by $T_1, \dots, T_{n-1}$ and $P_1$ implies $J_i'$ is indeed a two-sided ideal in $\mH$. Thus the theorem is proved. 
\end{proof}

By Theorem \ref{thm:Ji-basis}, we can easily obtain the following result.
\begin{cor}\label{cor:basis-Jk-Jk+1}
The set of images $\{\widetilde{T}_{(A,B,w)} + J_{i+1}\mid (A,B,w) \in \Gamma_i\}$ under the canonical projection $J_i\rightarrow J_i/J_{i+1}$ is a basis of $J_i/J_{i+1}$ for each $0\leq i\leq n$.
\end{cor}

}

%%%%%%%%%
%%%%%%%%%
%%%%%==================================================================================================
\section{Cellular bases of $\mH$}\label{sec:cellular}
%%%%%==================================================================================================
In this section, we will prove the cellularity of $\mH$ by constructing a cellular basis in the sense of \cite{GL96}.  We also show a Morita equivalence between $\mH$  and direct sum of the Hecke algebras associated to symmetric groups.

\subsection{Several key propositions for the ideals $J_k$}

Suppose $(A, B, w),(C, D, w')\in \Gamma_{k}$ with $0\leq k\leq n$. Recall $P_{k} T_{w} = T_{w} P_{k}$ due to \eqref{eq:Pcomm} and \eqref{eq:Pminus}. Hence 
\begin{align*}
\wtT_{(A,B,w)} \wtT_{(C,D,w')} &= \sigma\left(T_{A}^{-1}\right) P_{k} T_{w} T_{B}^{-1} \sigma\left(T_{C}^{-1}\right) P_{k} T_{w'} T_{D}^{-1} \\
&= \sigma\left(T_{A}^{-1}\right) T_{w} X_{B,C} T_{w'} T_{D}^{-1},
\end{align*}
where
\begin{equation}\label{eq:X-BC}
X_{B, C} = P_{k} T_{B}^{-1} \sigma\left(T_{C}^{-1}\right) P_{k}.
\end{equation}

\begin{prop}\label{prop:XBCq}
Suppose $B,C\in\mathscr{F}^n_k$ with $0\leq k\leq n$. Then 
$$
X_{B, C} \equiv
\begin{cases}
q^{\alpha_{B}} P_{k} \pmod{J_{k+1}} & \text{if } B = C, \\
0 \pmod{J_{k+1}}  & \text{if } B \neq C,
\end{cases}
$$
where, for $B = \{b_{1}<b_{2}< \cdots<b_{k}\}$, $\alpha_{B} = -\sum_{a=1}^{k} b_{a} + \frac{k(k+1)}{2}$.
\end{prop}

\begin{proof}
For any two sets $A, A'\in\mathscr{F}^n_k$ with $A=\{a_1<a_2<\cdots<a_k\}$ and $A'=\{a'_1<a'_2<\cdots<a'_k\}$, we write 
\begin{equation}\label{eq:lex-ord}
A<A'\text{ if } a_1=a'_1,\ldots,a_j=a'_j, a_{j+1}<a'_{j+1}
\end{equation}
for some $1\leq j\leq n-1$. This gives  the lexicographical order among the elements in $\mathscr{F}^n_k$.
By \eqref{eq:X-BC}, \eqref{eq:Pcomm} and \eqref{eq:Pminus} we can write $X_{B, C}$ in the following way:
\begin{equation}\label{eq:XBC-explicit}
\begin{aligned}
X_{B, C} = & P_{k}\left(T_{k}^{-1} T_{k+1}^{-1} \cdots T_{b_{k}-1}^{-1}\right) P_{k-1}\left(T_{k-1}^{-1} T_{k}^{-1} \cdots T_{b_{k-1}-1}^{-1}\right) P_{k-2} \cdots \\
& \cdots P_{1}\left(T_{1}^{-1} T_{2}^{-1} \cdots T_{b_{1}-1}^{-1}\right)\left(T_{c_{1}-1}^{-1} \cdots T_{2}^{-1} T_{1}^{-1}\right) P_{1} \cdots \\
& \cdots P_{k-2}\left(T_{c_{k-1}-1}^{-1} \cdots T_{k}^{-1} T_{k-1}^{-1}\right) P_{k-1}\left(T_{c_{k}-1}^{-1} \cdots T_{k+1}^{-1} T_{k}^{-1}\right) P_{k}, 
\end{aligned}
\end{equation}
where $B=\{b_1<b_2\cdots<b_k\}$ and $C=\{c_1<c_2\cdots<c_k\}$.

\textbf{Case 1}: Assume $B\neq C$. We shall show $X_{B,C} \equiv 0 \pmod{J_{k+1}}$ by induction using the lexicographic order on $B$.

Assume $B_0=\{1, \dots, k\}$ is the minimal element with respect to the \eqref{eq:lex-ord} in $\mathscr{F}^n_k$.  Here $T_B^{-1} = 1$, so by \eqref{eq:XBC-explicit} 
\begin{equation}\label{eq:XBC-further}
X_{B_0,C} = P_k (T_{c_1-1}^{-1} \cdots T_1^{-1}) P_1 (T_{c_2-1}^{-1} \cdots T_2^{-1}) P_2 \cdots P_{k-1} (T_{c_k-1}^{-1} \cdots T_k^{-1}) P_k,
\end{equation} 
with $C \neq B_0$.
We prove the right hand side of \eqref{eq:XBC-further} vanishes by induction using the lexicographic order on $C$. 
The smallest $C$ such that $C\neq B_0$ is $C_{\text{min}} = \{1,\dots,k-1,k+1\}$, for which $X_{B_0,C_{\text{min}}} = P_k T_k^{-1} P_k = -P_{k+1} \equiv 0 \pmod{J_{k+1}}$ by \eqref{eq:XBC-further} and \eqref{eq:Prec}. 
So the right hand side of \eqref{eq:XBC-further} vanishes in the case $C=C_{\text{min}}$. 
Now assume it vanishes for all $C' < C$ ($C' \neq B$). Suppose $u$ is the smallest positive integer such that $c_i=i$ for $1\leq i\leq u-1$ and $c_u>u$. Then by \eqref{eq:XBC-further} we have 
\begin{equation}\label{eq:XBC-further-2}
X_{B_0,C} = P_k (T_{c_u-1}^{-1} \cdots T_u^{-1}) P_u (T_{c_{u+1}-1}^{-1} \cdots T_{u+1}^{-1}) P_{u+1} \cdots P_{k-1} (T_{c_k-1}^{-1} \cdots T_k^{-1}) P_k,
\end{equation} 

If  $c_u-1\neq k$,  then $T_{c_u-1}^{-1}$ commutes with $P_k$ by \eqref{eq:Pcomm} and \eqref{eq:Pminus}, so by \eqref{eq:XBC-further-2} we obtain 
   \begin{equation}\label{eq:XBC-further-3}
   \begin{aligned}
&X_{B_0,C} \\
&=T_{c_u-1}^{-1} P_k (T_{c_u-2}^{-1} \cdots T_u^{-1}) P_u (T_{c_{u+1}-1}^{-1} \cdots T_{u+1}^{-1}) P_{u+1} \cdots P_{k-1} (T_{c_k-1}^{-1} \cdots T_k^{-1}) P_k\\
&=T_{c_u-1}^{-1}  X_{B_0,C'}
\end{aligned}
\end{equation}  
with $C'=\{c'_1<c'_2<\cdots<c'_{u-1}<c_u'<\cdots<c_k'\}$ such that $c_u'=c_u-1$ and $c_l'=c_l$ for $l\neq u$. Clearly $C'<C$ and thus by induction $X_{B_0,C'}=0\text{ mod }J_{k+1}$. Therefore by \eqref{eq:XBC-further-3} 
we have $X_{B_0,C}=0\text{ mod }J_{k+1}$ as $J_{k+1}$ is a two-sided ideal in $\mH$. 
    
 If $c_u-1= k$, then by \eqref{eq:Pcomm} and \eqref{eq:Pminus} as well as \eqref{eq:Prec} we have 
 \[
          \begin{aligned}
          P_k T_{c_u-1}^{-1} \cdots T_u^{-1} P_u
          &= P_k T_k^{-1} \cdots T_{u+1}^{-1} P_u T_u^{-1} P_u \\
          &= -P_k T_k^{-1} \cdots T_{u+1}^{-1} P_{u+1} \\
          &= -P_k T_k^{-1} \cdots T_{u+2}^{-1} P_{u+1} T_{u+1}^{-1} P_{u+1} \\
          &= (-1)^2 P_k T_k^{-1} \cdots T_{u+2}^{-1} P_{u+2} \\
          &= \cdots \\
          &= (-1)^{k-u} P_k T_k^{-1} P_k
          = (-1)^{k-u+1} P_{k+1}.
          \end{aligned}
          \]
This together with \eqref{eq:XBC-further-2} implies 
$$
X_{B_0,C}=(-1)^{k-u+1} P_{k+1}(T_{c_{u+1}-1}^{-1} \cdots T_{u+1}^{-1}) P_{u+1} \cdots P_{k-1} (T_{c_k-1}^{-1} \cdots T_k^{-1}) P_k\in J_{k+1}. 
$$
Putting together we obtain 
$$
X_{B_0,C}\equiv 0\text{ mod } J_{k+1} \quad \text{ if }C\neq B_0. 
$$

By induction, we assume $X_{B',C}\equiv 0\text{ mod } J_{k+1}$ for all $B' < B$ with respect to the order \eqref{eq:lex-ord} and any $C \neq B'$.  Take an arbitrary $C \neq B$.

If $C < B$, then $\sigma(X_{B,C}) = X_{C,B}$ by \eqref{eq:XBC-explicit}. Since $C< B$, by induction we have $X_{C,B} \equiv 0\text{ mod }J_{k+1}$, and  hence $X_{B,C} \equiv 0 \pmod{J_{k+1}}$ due to the fact $\sigma(J_{k+1})=J_{k+1}$.

If $C > B$, similarly there exists unique $s$ such that $b_a=a$ for $1\leq a\leq s-1$ and $b_s=i+1>s$ for some $i\geq 1$. Then we can write  $T_B^{-1} = T_{B'}^{-1} T_i^{-1}$ with $B' < B$ and $B'$ is obtained via $B$ be replacing $b_s=i+1$ with $b_s-1=i$. Hence by \eqref{eq:XBC-explicit} we have 
\begin{equation}\label{eq:XBC-further-4}
X_{B,C} = P_k T_{B'}^{-1} \bigl( T_i^{-1} \sigma(T_C^{-1}) P_k \bigr).
\end{equation}
We shall compute the term $T_i^{-1} \sigma(T_C^{-1}) P_k$ inside the right hand side of \eqref{eq:XBC-further-4}.  Observe $T_i^{-1} \sigma(T_C^{-1}) P_k=\sigma(P_kT_C^{-1}T_i^{-1})$.  
By the computation for $\wtT_{(A,B,w)}T^{-1}_{i}$ in the proof of Proposition \ref{pro:basis}, we have the following four cases: 

(1) $i,i+1\in C$. Then $P_kT_C^{-1}T_i^{-1}=-P_kT_C^{-1}$ by the computation in the proof of Proposition \ref{pro:basis}  and hence $T_i^{-1} \sigma(T_C^{-1}) P_k=-\sigma(T_C^{-1})P_k$. This together with \eqref{eq:XBC-further-4} implies 
$X_{B,C} =-P_kT_{B'}^{-1}\sigma(T_C^{-1})P_k=-X_{B',C}\equiv 0\text{ mod }J_{k+1}$ by induction since $C>B>B'$. 
 
(2) $i,i+1\notin C$. The again by the computation in the proof of Proposition \ref{pro:basis} we obtain $T_C^{-1}T_i^{-1}=T_j^{-1}T_{C}^{-1}$ for some $j>k$. Thus $T_i^{-1}\sigma(T_C^{-1})P_k=\sigma(P_kT_C^{-1}T_i^{-1})=\sigma(P_kT_j^{-1}T_C^{-1})= \sigma(T_C^{-1})P_k\,T_j^{-1}$. Then by \eqref{eq:XBC-further-4} we have  $
          X_{B,C}=P_k T_{B'}^{-1}\sigma(T_C^{-1})P_k\,T_j^{-1}=X_{B',C}\,T_j^{-1}
=\equiv 0\text{ mod }J_{k+1}$ by induction since $C>B>B'$. 
        
(3)  $i\in C$, $i+1\notin C$. Then $T_C^{-1}T_{i}^{-1}=T^{-1}_{C'}$ where $C'$ is obtained via $C$ by replacing  $i$ with $i+1$ and hence $X_{B,C} = P_k T_{B'}^{-1} \bigl( \sigma(T_{C'}^{-1}) P_k \bigr)=X_{B',C'}=\equiv 0\text{ mod }J_{k+1}$ by induction since $C'>C>B>B'$. .

(4)  $i\notin C$, $i+1\in C$. Then 
\begin{equation}\label{eq:XBC-further-5}
    T^{-1}_CT_i^{-1}  =(q^{-1}-1)T_C^{-1}+q^{-1}T_{C'}^{-1},
   \end{equation}
    where $C'$ is obtained from $C$ with $i+1$ replaced by $i$.
          Write $B=\{b_1<\cdots<b_k\}$, $C=\{c_1<\cdots<c_k\}$.
          By construction $B'$ is $B$ with $i+1$ replaced by $i$, so $i+1\in B$ and $i\notin B$.
          Recall that $s$ be the index with $b_s=i+1$; denote by $r$ the index with $c_r=i+1$.
          Since $C>B$, let $p$ be the smallest index with $c_p>b_p$  and $c_j=b_j$ for $1\leq j<p$.

          \emph{If $s<p$}, then  $c_s=b_s=i+1$, hence $r=s$. The clearly $C'>B'$.

          \emph{If $s=p$}, then $b_p=i+1$, $b'_p=i$. As $c_p>b_p=i+1$ we have $c_p\ge i+2$, so $r\neq p$ and  hence $r>p$, otherwise $b_r=c_r=i+1=b_p$, contradicting distinctness. Thus we have $c'_j=b'_j$ for $j<p$; and $c'_p=c_p\ge i+2>i=b'_p$. Hence $C'>B'$.

          \emph{If $s>p$,} then $b'_p=b_p$, and  $c'_j=b'_j$ for $1\leq j<p$.
          If $r=p$, then $c_p=i+1$, $c'_p=i$, and $b_p<c_p=i+1$. Meanwhile $i\notin B$, thus  $b_p\le i-1$. This means $c'_p=i>b_p=b'_p$.
          If $r\neq p$, then $c'_p=c_p>b_p=b'_p$.
          Hence $C'>B'$.
Putting together and by \eqref{eq:XBC-further-5} as well as \eqref{eq:XBC-further-4}  we obtain 
$X_{B,C} = P_k T_{B'}^{-1} \bigl( \sigma(T_{C}^{-1}) P_k \bigr)=(q^{-1}-1)X_{B',C}+q^{-1}X_{B',C'}\equiv 0\text{ mod }J_{k+1}$ by induction since $C>C'>B'$ and $B>B'$.

\textbf{Case 2}: Assume $B=C$. We shall $X_{B,B} \equiv q^{\alpha_B} P_k \pmod{J_{k+1}}$.
If $B = \{1,\dots,k\}$ then $\alpha_B = 0$ and $X_{B,B} = P_k = q^0 P_k$. Otherwise as above there exists a unique $s$ such that $b_a=a$ for $1\leq a\leq s-1$ and $b_s=i+1>s$. Then 
by \eqref{eq:XBC-explicit} we have 
$$
\begin{aligned}
X_{B,B}&=P_k T_B^{-1} \sigma(T_B^{-1}) P_k
=P_k \cdots T_i^{-1} T_i^{-1} \cdots P_k\\
&=q^{-1} X_{B',B'} - q^{-1}(q-1) X_{B,B'}\equiv q^{-1} X_{B',B'} \text{ mod }J_{k+1}
\end{aligned}
$$
due to the fact that $T_i^{-2} = q^{-1} - q^{-1}(q-1)T_i^{-1}$ and $X_{B,B'} \equiv 0 \pmod{J_{k+1}}$ since $B \neq B'$, where $B'$ is obtained from $B$ by replacing $i+1$ by $i$.
Repeating this process we eventually $B_0 = \{1,\dots,k\}$ with $X_{B_0,B_0} = P_k$.  Each step multiplies by $q^{-1}$, decreasing $\alpha$ by $1$.  By induction,
\[
X_{B,B} \equiv q^{\alpha_B} P_k \pmod{J_{k+1}}, 
\]
with $-\sum_{i=s}^k(b_i-i)=-\sum_{i=1}^k(b_i-i)=-\sum_{i=1}^k b_i+\frac{k(k+1)}{2}=\alpha_B$ as noting $b_i-i=0$ for $1\leq i\leq s-1$ by the choice of $s$. 
Thus the proposition is proved.
\end{proof}

\begin{prop}\label{prop:layer-mult-basis}
Let $(A, B, w),(C, D, w') \in \Gamma_{k}$ with $0\leq k\leq n$. Then 
$$
\wtT_{(A,B,w)} \wtT_{(C,D,w')}\equiv
\begin{cases}
q^{\alpha_{B}} \sum_{u\in\mathfrak{S}'_{n-k}} \mathtt{r}^{u}_{w,w'} \wtT_{(A,D,u)}  \pmod{J_{k+1}} & \text{if } B = C, \\
0 \pmod{J_{k+1}}  & \text{if } B \neq C,
\end{cases}
$$
where $T_{w} T_{w'} = \sum_{u \in \mathfrak{S}_{n-k}'} \mathtt{r}^{u}_{w,w'} T_{u}\in\Hta$.
\end{prop}
\begin{proof}
By \eqref{eq:X-BC} we have 
\begin{equation}\label{eq:wtT-wtT}
\wtT_{(A,B,w)} \wtT_{(C,D,w')} = \sigma(T_A^{-1}) T_w X_{B,C} T_{w'} T_D^{-1}
\end{equation}
Then by Proposition~\ref{prop:XBCq}  $\wtT_{(A,B,w)} \wtT_{(C,D,w')} \equiv 0 \pmod{J_{k+1}}$ if $B \neq C$ as $J_{k+1}$ is a two-sided ideal. 
If $B = C$, then clearly $P_k$ commutes with $T_w$ and $T_{w'}$, so by \eqref{eq:wtT-wtT} and Proposition~\ref{prop:XBCq} we obtain 
\[
\wtT_{(A,B,w)} \wtT_{(C,D,w')}
\equiv \sigma(T_A^{-1}) T_w (q^{\alpha_B} P_k) T_{w'} T_D^{-1}
= q^{\alpha_B} \sigma(T_A^{-1}) P_k (T_w T_{w'}) T_D^{-1} \pmod{J_{k+1}}.
\]
Writing $T_w T_{w'} = \sum_{\sigma\in\mathfrak{S}'_{n-k}} \gamma^\sigma_{w,w'} T_\sigma$ in $\Hta$, we get
\[
\wtT_{(A,B,w)} \wtT_{(C,D,w')} \equiv q^{\alpha_B} \sum_{\sigma\in\mathfrak{S}'_{n-k}} \gamma^\sigma_{w,w'} \,\wtT_{(A,D,\sigma)} \pmod{J_{k+1}}.
\]
This proves the theorem. 
\end{proof}

Below we derive an important property that assists us in proving the cellular basis.

\begin{prop}\label{prop:colstab}
Let $0\le k\le n$.  For each $x\in\mH$ and $A\in\mathscr{F}^n_k$, there are a family of unique elements $h^{(k)}_{C,A}(x)\in\mathcal{H}'_{n-k,R}(q)$ with $C\in\mathscr{F}^n_k$, such that
\begin{equation}\label{eq:h-definition}
  x\,\sigma(T_A^{-1})P_k
  \equiv
  \sum_{C\in\mathscr{F}^n_k} \sigma(T_C^{-1})\, P_k\, h^{(k)}_{C,A}(x)
  \pmod{J_{k+1}}.
\end{equation}
Moreover 
\begin{equation}\label{eq:h-multiplicative}
  h^{(k)}_{C,A}(y_1y_2)=\sum_{B\in\mathscr{F}^n_k}h^{(k)}_{C,B}(y_1)h^{(k)}_{B,A}(y_2).
\end{equation}
for any $y_1,y_2\in\mH$. 
\end{prop}
\begin{proof}
Observe that $x\,\sigma(T_A^{-1})P_k\in J_k$. Then we can write 
    \begin{equation}\label{eq:key-lem-for1}
    x\,\sigma(T_A^{-1})P_k \equiv \sum_{C,D\in\mathscr{F}^n_k}\sum_{w\in \Sn'_{n-k}}
  f_{C,D,w}\wtT_{(C,D,w)}\pmod{J_{k+1}}.
  \end{equation}
  for some $f_{C,D,w}\in R$
  by Corollary \ref{cor:basis-Jk-Jk+1}. 
  Note that $P_k=\wtT_{([k],[k],\text{id})}$, where we recall $[k]=\{1<2<\cdots<k\}$. Then  by Proposition \ref{prop:layer-mult-basis} we obtain 
 \begin{equation}\label{eq:key-lem-for2}
 \begin{aligned}
x\,\sigma(T_A^{-1})P_k= x\,\sigma(T_A^{-1})P_k P_k&\equiv \sum_{C,D\in\mathscr{F}^n_k}\sum_{w\in \Sn'_{n-k}}
  f_{C,D,w}\wtT_{(C,D,w)}\wtT_{([k],[k],\text{id})}\\
  & \equiv
  \sum_{C,w}f_{C,[k],w}\wtT_{(C,[k],w)}\pmod{J_{k+1}}. 
  \end{aligned}
  \end{equation}
Set $h^{(k)}_{C,A}(x)=\sum_{w\in\mathfrak{S}'_{n-k}}f_{C,[k],w}T_w\in\mathcal{H}'_{n-k,R}(q)$, while the uniqueness $h^{(k)}_{C,A}(x)$ follows from the linear independence of the elements $\wtT_{(C,[k],w)}$ with $C\in\mathscr{F}^n_k$ and $w\in \Sn'_{n-k}$ in $J_k/J_{k+1}$.

Suppose $y_1,y_2\in\mH$. By \eqref{eq:h-definition}, we have 
$
y_2\,\sigma(T_A^{-1})P_k
  \equiv
  \sum_{C\in\mathscr{F}^n_k} \sigma(T_C^{-1})\, P_k\, h^{(k)}_{C,A}(y_2)
  \pmod{J_{k+1}}
$
and hence 
$$
\begin{aligned}
y_1y_2\,\sigma(T_A^{-1})P_k
  \equiv
  \sum_{B\in\mathscr{F}^n_k}y_1 \sigma(T_B^{-1})\, P_k\, h^{(k)}_{B,A}(y_2)
  \pmod{J_{k+1}}\\
   \equiv
  \sum_{C,B\in\mathscr{F}^n_k} \sigma(T_C^{-1})\, P_kh^{(k)}_{C,B}(y_1)\, h^{(k)}_{B,A}(y_2)  \pmod{J_{k+1}}
  \end{aligned}
$$
Meanwhile applying \eqref{eq:h-definition} to the elements $x=y_1y_2$, the uniqueness of $h^{(k)}_{C,A}(y_1y_2)$ gives rise to \eqref{eq:h-multiplicative}. This proves the lemma. 
\end{proof}

\subsection{Cellular bases for $\mH$}

First, we shall recall the definition of cellular algebra in the sense of \cite{GL96}.
\begin{definition}\cite[Definition 1.1]{GL96}\label{defn:GL96}
Let $R$ be a commutative ring with identity. An associative $R$-algebra $A$ is called a cellular algebra if it possesses a \textit{cell datum} $(\Lambda, M, C, *)$, where:
\begin{enumerate}
    \item $\Lambda$ is a finite partially ordered set (poset) with order $\prec$, and for each $\lambda \in \Lambda$, $M(\lambda)$ is a finite index set.
    \item $C$ is an injective map 
    $$\coprod_{\lambda \in \Lambda} M(\lambda) \times M(\lambda) \longrightarrow A$$
    whose image forms an $R$-basis of $A$. The basis element corresponding to a pair $(\mathfrak{s}, \mathfrak{t}) \in M(\lambda) \times M(\lambda)$ is denoted by $C_{\mathfrak{s}, \mathfrak{t}}^\lambda$.
    \item $*$ is an $R$-linear anti-automorphism of $A$ (i.e., $(ab)^* = b^* a^*$) such that 
    $$(C_{\mathfrak{s}, \mathfrak{t}}^\lambda)^* = C_{\mathfrak{t}, \mathfrak{s}}^\lambda$$
    for all $\lambda \in \Lambda$ and $\mathfrak{s}, \mathfrak{t} \in M(\lambda)$.
    \item For any $a \in A$ and $\lambda \in \Lambda, \mathfrak{s}, \mathfrak{t} \in M(\lambda)$, the multiplication satisfies the following upper-triangular property:
    $$a\cdot C_{\mathfrak{s}, \mathfrak{t}}^\lambda  \equiv \sum_{\mathfrak{v} \in M(\lambda)} r_a(\mathfrak{s}, \mathfrak{v}) C_{\mathfrak{v}, \mathfrak{t}}^\lambda \pmod{A_{\succ \lambda}},$$
    where the scalar coefficient $r_a(\mathfrak{s}, \mathfrak{v}) \in R$ depends only on $a$, $\mathfrak{s}$, and $\mathfrak{v}$ (independent of $\mathfrak{t}$). The space $A_{\succ\lambda}$ is the $R$-submodule of $A$ spanned by the set:
    $$\{ C_{\mathfrak{u}, \mathfrak{v}}^\mu \mid \mu \succ \lambda \text{ and } \mathfrak{u}, \mathfrak{v} \in M(\mu) \}.$$
\end{enumerate}
\end{definition}

Fix $0\le k\le n$. For notation simplicity we write $\mathscr{H}_n:=\mH$ and 
\begin{equation}\label{eq:tail}
\Hta:=\mathcal{H}'_{n-k,R}(q)=\langle T_{k+1},T_{k+2},\dots,T_{n-1}\rangle\subset\mathcal{H}_{n,R}(q)\subset\mH. 
\end{equation} Clearly$\mathcal H'_0=R$, and $\Hta$  is isomorphic to the Hecke algebra $\mathcal{H}_{n-k,R}(q)$. 
It is known that $\mathcal{H}'_{n-k}$ is a cellular algebra \cite[Example 1.2]{GL96}. 
Assume $(\Lambda_k, M_k,\{c^\lambda_{\ks,\kt}\},\sigma_k)$ is a cell datum for each $\mathcal{H}'_{n-k}$ with $0\leq k\leq n$ as described in Definition \ref{defn:GL96} such that
 \begin{equation}\label{eq:two-anti}
\sigma_k=\sigma |_{\mathcal{H}'_{n-k}}, 
\end{equation}
that is, the restriction of the anti-involutation $\sigma$ to the Hecke subalgebras  $\mathcal{H}'_{n-k}$ coincides with the anti-involution $\sigma$. Here $c^\lambda_{\mathfrak{s},\mathfrak{t}}$ for $(\mathfrak{s},\mathfrak{t})\in M_k(\lambda)\times M_k(\lambda)$ and $\lambda\in\Lambda_k$ is the corresponding cellular basis.  Clearly both Murphy basis and  the Kazhdan-Lusztig basis are such a basis for Hecke algebras, cf. \cite[Example 1.2]{GL96} or \cite{Ma99}.  Denote by the partial order in $\Lambda_k$ by $\prec_k$. 

Moreover by Definition \ref{defn:GL96}  we have 
\begin{equation}\label{eq:Hecke-involution}
\sigma_k(c^{\lambda}_{{\ks},{\kt}})=c^{\lambda}_{{\kt},{\ks}}.
\end{equation}
Write $(\Hta)^{\succ_k\lambda}$ for the $R$-submodule of $\Hta$ spanned by the cellular basis elements $c^{\mu}_{{\ku},{\kv}}$ with $\mu\succ_k\lambda$.

\begin{definition}\label{def:cellular-basis}
Set
$$\Lambda=\{(k,\lambda):0\le k\le n,\;\lambda\in \Lambda_k\}.$$
Define a partial order in $\Lambda$ via 
\begin{equation}\label{eq:poset}
  (k,\lambda)\prec (l,\mu)
  \quad\Longleftrightarrow\quad
  k<l\quad\text{or}\quad k=l\text{ and }\lambda\prec_k\mu.
\end{equation}
For $(k,\lambda)\in\Lambda$, set
$$\mathcal{M}(k,\lambda):=\binnk\times M_k(\lambda).$$
For $(A,\ks),(B,\kt)\in\mathcal{M}(k,\lambda)$, define
\begin{equation}\label{eq:C-basis}
  C^{(k,\lambda)}_{(A,\ks),(B,\kt)}
  :=\sigma(T_A^{-1})P_k\,c^{\lambda}_{\ks,\kt}\,T_B^{-1}.
\end{equation}
\end{definition}

\begin{lem}\label{lem:new-basis}
The set $\{C^{(k,\lambda)}_{(A,\ks),(B,\kt)}\mid (k,\lambda)\in\Lambda, (A,\ks),(B,\kt)\in \mathcal{M}(k,\lambda)\}$ is a R-basis for $\mH$. 
\end{lem}
\begin{proof}
By \eqref{eq:tilde-basis} and the fact that $\{c^\lambda_{\ks,\kt}\mid \lambda\in\Lambda_k, \ks,\kt\in M_i(\lambda)\}$ is a basis for $\mathcal{H}'_{n-i}$ for each $0\leq k\leq n$, we obtain that $\wtT_{(A,B,w)}$ for each $(A,B,w)\in\Gamma_k$ is a linear combination of elements $C^{(i,\mu)}_{\ku,\kv}$ with $\mu\in\Lambda_k, \ku,\kv\in M_k(\lambda)$. Therefore by Proposition \ref{pro:basis}, $\mH$ is spanned by the set $\{C^{(k,\lambda)}_{(A,\ks),(B,\kt)}\mid (k,\lambda)\in\Lambda, (A,\ks),(B,\kt)\in \mathcal{M}(k,\lambda)\}$. 
Then the lemma is verified by a cardinality comparison by Proposition \ref{pro:basis}. 
\end{proof}
Let $\mathscr{H}_n^{\succ(k,\lambda)}$ denote the $R$-span of all basis elements
$C^{(l,\mu)}_{(A,\ku),(B,\kv)}$ with $(l,
\mu)\succ(k,\lambda)$ for each $(k,\lambda)\in\Lambda$. 

\begin{lem} 
For $0\leq k\leq n$, we have 
\begin{equation}\label{eq:dominance-inclusion}
J_{k+1}\subseteq\mathscr{H}_n^{\succ(k,\lambda)},\quad   \sigma(T_A^{-1})P_k\, y\, T_B^{-1}\in \mathscr{H}^{\succ(k,\lambda)}, \quad \text{ for any } y\in (\Hta)^{\succ\lambda}. 
\end{equation}
\end{lem}
\begin{proof}
Fix $0\leq k\leq n$. 
By Theorem \ref{thm:Ji-basis}, it suffices to prover $\wtT_{(A,B,w)}\in \mathscr{H}_n^{\succ(k,\lambda)}$ for each $(A,B,w)\in \cup_{i=k+1}^n\Gamma_i$ to verify $J_{k+1}\subseteq\mathscr{H}_n^{\succ(k,\lambda)}$.  
As in the proof of Lemma \ref{lem:new-basis}  the element $\wtT_{(A,B,w)}$ for each $(A,B,w)\in\Gamma_i$ is a linear combination of elements $C^{(i,\mu)}_{\ku,\kv}$ with $\mu\in\Lambda_i, \ku,\kv\in M_i(\mu)$. Meanwhile $C^{(i,\mu)}_{\ks,\kt}\in \mathscr{H}_n^{\succ(k,\lambda)}$ for each $i\geq k+1$ due to \eqref{eq:poset}.  This proves $J_{k+1}\subseteq\mathscr{H}_n^{\succ(k,\lambda)}$.  

Take $y\in (\Hta)^{\succ\lambda}$. Then $y$ is a linear combination of cellular basis $c^{\mu}_{{\ku},{\kv}}$ with $\mu\succ_k\lambda$ by the definition of $ (\Hta)^{\succ\lambda}$. This means $\sigma(T_A^{-1})P_k\, y\, T_B^{-1}$ is a linear combination of 
$\sigma(T_A^{-1})P_k\,c^{\mu}_{{\ku},{\kv}}\, T_B^{-1}=C^{(k,\mu)}_{\ku,\kv}$ with $\mu\succ_k\lambda$. Then by \eqref{eq:poset} we obtain $\sigma(T_A^{-1})P_k\, y\, T_B^{-1}\in \mathscr{H}^{\succ(k,\lambda)}$.  This proves the lemma. 

\end{proof}

\begin{thm}\label{thm:cellularity}
The quadruple
$$\bigl(\Lambda,\mathcal{M},\{C^{(k,\lambda)}_{(A,\ks),(B,\kt)}\},\sigma\bigr)$$
is a Graham--Lehrer cellular datum for $\mH$.  In particular, $\mH$ is cellular.
\end{thm}

\begin{proof}
   By Lemma \ref{lem:new-basis}, it remains to show that the quadruple
$\bigl(\Lambda,\mathcal{M},\{C^{(k,\lambda)}_{(A,\ks),(B,\kt)}\},\sigma\bigr)$ satisfies (3) and (4) in Definition \ref{defn:GL96}. 

 Observe that  
  Since $P_k$ commutes with $\Hta$ due to \eqref{eq:Pcomm} and \eqref{eq:Pminus}, we have
        $$
        \sigma\bigl(C^{(k,\lambda)}_{(A,\ks),(B,\kt)}\bigr)=\sigma(T_B^{-1})\sigma(c^{\lambda}_{\ks,\kt})\,P_k\,T_A^{-1}
  =\sigma(T_B^{-1})P_k\, c^{\lambda}_{\kt, \ks}T_A^{-1}
  =C^{(k,\lambda)}_{(B,\kt),(A,\ks)}
  $$
 using \eqref{eq:Hecke-involution} and \eqref{eq:two-anti} as well as \eqref{eq:C-basis}.   So part (3) in Definition \ref{defn:GL96} holds for the quadruple
$\bigl(\Lambda,\mathcal{M},\{C^{(k,\lambda)}_{(A,\ks),(B,\kt)}\},\sigma\bigr)$. 

  It remains to check the cellular triangularity axiom.  Let $x\in\mH$.  By Proposition \ref{prop:colstab} and the fact $J_{k+1}$ is a two-sided ideal, we have 
  \begin{equation}\label{eq:basis-triangularity-1}
  x\,C^{(k,\lambda)}_{(A,\ks),(B,\kt)}=x\,\sigma(T_A^{-1})P_k\,c^{\lambda}_{\ks,\kt}T_B^{-1}
  \equiv \sum_{C\in\binnk}\sigma(T_C^{-1})P_k\,h^{(k)}_{C,A}(x)\,c^{\lambda}_{\ks,\kt}T_B^{-1}
  \pmod{J_{k+1}}.
  \end{equation}
  As $h^{(k)}_{C,A}(x)\in \mathcal{H}'_{n-k}$, then by choice of the cell datum $(\Lambda_k, M_k, \{c^\lambda_{\ks,\kt}\},\sigma_k)$ for $\mathcal{H}'_{n-k}$ we have the following 
  \begin{equation}\label{eq:Hecke-basis-triangularity}
  h^{(k)}_{C,A}(x)\, c^{\lambda}_{\ks,\kt}
  \equiv
  \sum_{\ku\in M_k(\lambda)}\mathtt{r}_x(C,\ku;A,\ks)\, c^{\lambda}_{\ku, \kt}
  \pmod{(\Hta)^{\succ\lambda}}.
\end{equation}
for some $\mathtt{r}_x(C,\ku;A,\ks)\in R$.  
Putting \eqref{eq:basis-triangularity-1}, \eqref{eq:Hecke-basis-triangularity} and \eqref{eq:dominance-inclusion} together one can get 
 \begin{equation}\label{eq:mirabolic-basis-triangularity}
x\, C^{(k,\lambda)}_{(A,\ks),(B,\kt)}
  \equiv
  \sum_{C\in\binnk}\sum_{\ku\in\Std(\lambda)}
  \mathtt{r}_x(C,\ku;A,\ks)
  C^{(k,\lambda)}_{(C,\ku),(B,\kt)}
  \pmod{\mathscr{H}_n^{\succ(k,\lambda)}}.
\end{equation}
This proves that part (4)  in Definition \ref{defn:GL96} holds for the quadruple
$\bigl(\Lambda,\mathcal{M},\{C^{(k,\lambda)}_{(A,\ks),(B,\kt)}\},\sigma\bigr)$. 
Putting together, the theorem is proved. 
\end{proof}

\begin{rmk}
In \cite{Ro14}, Rosso introduced a set of elements by identifying the mirabolic Hecke algebra over $\mathbb{C}(q)$ with a quotient of a certain cyclotomic Hecke algebra. He conjectured that these elements form a cellular basis and, consequently, that the mirabolic Hecke algebra over $\mathbb{C}(q)$ is cellular; see \cite[Conjecture 7.5]{Ro14}. Theorem \ref{thm:cellularity} establishes the cellularity assertion of this conjecture. However, whether the proposed elements form a cellular basis remains open. It would therefore be interesting to compare them with the cellular bases constructed in Theorem \ref{thm:cellularity}.
\end{rmk}
%%%
\begin{rmk}\label{rmk:rook-correction}
König and Xi \cite{KX99} introduced the notion of an iterated inflation of cellular algebras, and Xi \cite[Lemma 3.3]{Xi99} subsequently proposed a precise characterization of such inflations. The calculations in Propositions \ref{prop:XBCq} and \ref{prop:layer-mult-basis} parallel those in \cite{Pa06}, where Paget verified that the $q$-rook monoid algebras satisfy the conditions of \cite[Lemma 3.3]{Xi99} and used this to establish their cellularity. Similarly, Propositions \ref{prop:XBCq} and \ref{prop:layer-mult-basis} show that $\mH$ satisfies those conditions. However, Green and Paget \cite{GP18} later gave a counterexample to \cite[Lemma 3.3]{Xi99} and provided a corrected statement in \cite[Theorem 1]{GP18}. Verifying the conditions of the corrected theorem for either the $q$-rook monoid algebras or $\mH$ appears to be nontrivial. We instead establish the multiplication formula in Proposition \ref{prop:colstab}, which is a key ingredient in our proof of Theorem \ref{thm:cellularity}. We expect an analogous formula to hold for the $q$-rook monoid algebras.

The calculations in Propositions \ref{prop:XBCq} and \ref{prop:layer-mult-basis} also suggest that various formulas for mirabolic Hecke algebras are related to their counterparts for $q$-rook monoid algebras by interchanging the parameters $q$ and $-1$. On the Hecke algebra $\mathcal{H}_n(q)$, which occurs as a subalgebra of both $\mH$ and $\mathscr{R}_{n,R}(q)$, this interchange is realized by the Mullineux involution
$
T_i\longmapsto -qT_i^{-1},
\qquad 1\leq i\leq n-1.
 $
It is therefore natural to ask whether this involution extends to an isomorphism between the mirabolic Hecke algebra and a $q$-rook monoid algebra. Inspired by \cite{JL09}, we construct an explicit isomorphism over $R$, generalizing the isomorphism established over $\mathbb{C}(q)$ in \cite[Proposition 5.2]{Ro14}.

More precisely, the $q$-rook monoid algebra $\mathscr{R}_{n,R}(q^{-1})$ is the associative $R$-algebra generated by $\widetilde{P}_1,\ldots,\widetilde{P}_n$ and $\widetilde{T}_1,\ldots,\widetilde{T}_{n-1}$, subject to the relations in \cite[(2.1), Corollary 2.1]{Ha04}. A direct verification of the defining relations shows that the assignment
\begin{equation}
\begin{aligned}
\tau_R:\mH &\longrightarrow \mathscr{R}_{n,R}(q^{-1}),\\
P_j &\longmapsto \widetilde{P}_j,
&\quad T_i &\longmapsto -q\widetilde{T}_i,
\end{aligned}
\end{equation}
for $1\leq j\leq n$ and $1\leq i\leq n-1$, defines an isomorphism. Its inverse is given by
\[
\tau_R^{-1}(\widetilde{P}_j)=P_j,
\qquad
\tau_R^{-1}(\widetilde{T}_i)=-q^{-1}T_i.
\]
Transporting the cellular bases from Theorem \ref{thm:cellularity} along this isomorphism yields cellular bases for the $q$-rook monoid algebras. This gives a direct proof of their cellularity that avoids reliance on \cite[Lemma 3.3]{Xi99}, in light of the correction in \cite{GP18}.

We expect the cellular bases constructed in Theorem \ref{thm:cellularity} to be useful in our study of Kazhdan--Lusztig-type canonical bases for both $q$-rook monoid algebras and mirabolic Hecke algebras, which we will pursue in \cite{LW26}.
\end{rmk}

\subsection{A Moita equivalence}

In this subsection, we assume $R=\mathbb{K}$ is a field with invertible $q\in\mathbb{K}$. 
   For $0\leq k\leq n$, let $V_k$ denote a $\mathbb{K}$-vector space with basis $\{v_A\mid A\in\mathscr{F}^n_k\}$ labelled by $\mathscr{F}^n_k$. 
   Write kronecker delta as 
\begin{equation}\label{eq:Phi}
	\delta_{B,C}=
	\begin{cases}
		1& \text{if } B = C, \\
		0 & \text{if } B \neq C.
	\end{cases}
\end{equation}
for any $B,C\in\mathscr{F}^n_k$. 
Define a multiplication on the $\mathbb{K}$-vector space $V_k\otimes_{\mathbb{K}} V_k\otimes_{\mathbb{K}}\mathcal{H}'_{n-k,\mathbb{K}}(q)$ via 
\begin{equation}\label{eq:multiplication-tensor}
	\left(v_{A} \otimes v_{B} \otimes T_{w}\right)\left(v_{C} \otimes v_{D} \otimes T_{w'}\right) =\delta_{B,C}q^{\alpha_B}\sum_{u\in\mathfrak{S}'_{n-k}}\mathtt{r}^u_{w,w'}  v_{A} \otimes v_{D} \otimes T_u
\end{equation}
for $A,B,C,D\in\mathscr{F}^n_k$ and $w,w'\in\mathfrak{S}_{n-k}$. 
%%%%%%

\begin{lem} \label{lem:matrix-iso}
For $0\leq k\leq n$,  we have the following algebra isomorphism:
	$$
	V_{k} \otimes V_{k} \otimes \mathcal{H}'_{n-k,\mathbb{K}}(q)\cong\Mat_{\binom{n}{k}}(\mathcal{H}'_{n-k,\mathbb{K}}(q)).
	$$
\end{lem}

\begin{proof}
As $|\mathscr{F}^n_k|=\binom{n}{k}$, we can parametrize the rows and columns of the matrices in $\Mat_{\binom{n}{k}}(\mathcal{H}'_{n-k,\mathbb{K}}(q))$ by the elements $A\in\Gamma_k$. For each $h\in \mathcal{H}'_{n-k,\mathbb{K}}(q)$, denote by 
$E_{A,B}(x)$ the matrix whose $(A,B)$-entry is $x$ and other entry are $0$. 
Then it is straightforward using the multiplication \eqref{eq:multiplication-tensor}  to prove that the map $v_A\otimes v_B\otimes h\mapsto E_{A,B}(q^{\alpha_B}h)$ for each $h\in\mathcal{H}'_{n-k,\mathbb{K}}(q)$ gives an algebra isomorphism
 from $V_{k} \otimes V_{k} \otimes \mathcal{H}'_{n-k,\mathbb{K}}(q)$ to $\Mat_{\binom{n}{k}}(\mathcal{H}'_{n-k,\mathbb{K}}(q)).$
\end{proof}
%%%%%
\begin{lem}\label{lem:psi-iso} 
For each $0\leq k\leq n$, the map 
	$$
	\psi_k: J_{k} / J_{k+1} \rightarrow V_{k} \otimes V_{k} \otimes \mathcal{H}'_{n-k,\mathbb{K}}(q),\quad \wtT_{(A,B,w)}+J_{k+1} \mapsto v_{A} \otimes v_{B} \otimes T_{w}
	$$
	is an algebra isomorphism. \end{lem}
\begin{proof} 
	It  follows directly from  Proposition~\ref{prop:layer-mult-basis} and we leave the details to the reader. 
\end{proof}

\begin{lem}\cite[Lemma~7.1]{KX01} \label{lem:ring-splitting}
Suppose $A$ is a ring and $I$ a 2-sided ideal in $A$. If there is a unit element in $I$, then  $A\cong I\oplus (A/I)$ as rings.  
\end{lem}

Let $R=\mathbb{K}$ be a field and $q\in\mathbb{K}$ is a primitive $e$-th root of unity, that is $e$ is the smallest positive integer such $1+q+\cdots+q^{e-1}=0$ and set $e=\infty$ if no such positive number. Denote by $\mathcal{P}(n)$ the set of partitions of $n$ for each $n\geq 0$. A partition $\lambda=(\lambda_1\geq \lambda_2\geq \ldots)\in\mathcal{P}(n)$ is said to be $e$-restricted if $\lambda_i-\lambda_{i+1}<e$.  
Denote by $\mathcal{RP}_e(n)$ the set of $e$-restricted partitions for each $n\geq 0$. 
Then, analogously to the proof in \cite[Theorem 14 \& Corollary 15]{Pa06} we have the following result.

\begin{thm}\label{thm:Morita_equi_MH}
There exists an algebra isomorphism 
	$$\mathscr{H}_{n,\mathbb{K}}(q)\cong \oplus_{k=0}^n\Mat_{\binom{n}{k}}(\mathcal{H}'_{n-k,\mathbb{K}}(q))$$ as an algebra. Moreover, the mirabolic Hecke algebra $\mathscr{H}_{n,\mathbb{K}}(q)$ is Morita equivalent to 
	$$
	\oplus_{k=0}^n\mathcal{H}_{k,\mathbb{K}}(q),
	$$
	where we write $\mathcal{H}_{0,\mathbb{K}}(q)=\mathbb{K}$. 
\end{thm}
\begin{proof}
By Lemmas \eqref{lem:matrix-iso}-\eqref{lem:psi-iso}, it suffices to prove that there exists a unit in $J_k/J_{k+1}$ for each $0\leq k\leq n$. Let $e_k=\sum_{A\in\mathscr{F}^n_k}q^{-\alpha_A}\wtT_{(A,A,\text{id})}+J_{k+1}\in J_k/J_{k+1}$. Then 
by Proposition \ref{prop:layer-mult-basis} one can easily show $e_k(\wtT_{(B,C,w)}+J_{k+1})=(\wtT_{(B,C,w)}+J_{k+1})e_k=\wtT_{(B,C,w)}+J_{k+1}$ for each $(B,C,w)\in\Gamma_k$. Hence $e_k$ is a unit in $J_k/J_{k+1}$. This proves the theorem. 
\end{proof}

\begin{cor}\label{cor:semisimplicity}
The mirabolic Hecke algebra $\mathscr{H}_{n,\mathbb{K}}(q)$ is semisimple if and only if $e>n$. Moreover, the isomorphism classes of irreducible $\mathscr{H}_{n,\mathbb{K}}(q)$-modules are parametrized by
\[
\{(k,\lambda)\mid 0\leq k\leq n,\ \lambda\in\mathcal{RP}_e(n-k)\}.
\]
\end{cor}

\begin{rmk}
By Theorem \ref{thm:Morita_equi_MH}, the representation theory of $\mHK$ reduces to that of the usual Hecke algebras. In addition to the semisimplicity criterion and the parametrization of irreducible modules in Corollary \ref{cor:semisimplicity}, known results for Hecke algebras also determine the decomposition numbers and block decomposition of $\mHK$.
\end{rmk}

%%%%%%%%%%%%%%%%%

\section{Coincidence between cell modules and Siegel's modules over $\mHK$}\label{sec:cell}

In this section, we study the cell modules and show the coincidence of these modules with Siegel's irreducible modules over $\mHK$ in 
the case $\mathbb{K}=\mathbb{C}(q)$. 

\subsection{cell modules of cellular algebra}
We will show some properties of cellular algebra in this section.

\begin{definition}\label{defn:usual-cell-module}
Let $A$ be a cellular algebra over a field $K$ with cellular basis $\{C_{{\ks},{\kt}}^\lambda \mid \lambda \in \Lambda, \ {\ks}, {\kt} \in M(\lambda)\}$. For each $\lambda \in \Lambda$, the \emph{cell module} (or standard module), denoted by $W(\lambda)$, is defined as a left $A$-module with a $K$-basis $\{C_{\ks} \mid {\ks} \in M(\lambda)\}$. The left action of $a \in A$ on $W(\lambda)$ is given by
$$a \cdot C_{\ks} = \sum_{{\kv} \in M(\lambda)} r_a({\kv}, {\ks}) C_{\kv},$$
where the scalars $r_a({\kv},{\ks}) \in K$ are uniquely determined by the cellular algebra multiplication rule $a C_{{\ks},{\kt}}^\lambda \equiv \sum_{{\kv} \in M(\lambda)} r_a({\kv}, {\ks}) C_{{\kv},{\kt}}^\lambda \pmod{A_{\succ\lambda}}$.
\end{definition}

\begin{definition}\label{defn:canonical form}
Let $A$ be a cellular algebra over a field $K$ with cellular basis $\{C_{{\ks},{\kt}}^\lambda \mid \lambda \in \Lambda, \ {\ks},{\kt} \in M(\lambda)\}$. For any $\lambda \in \Lambda$ and ${\ku},{\kv} \in M(\lambda)$, the cellular algebra axioms ensure that the multiplication of basis elements satisfies:
$$C_{{\ks},{\ku}}^\lambda C_{{\kv},{\kt}}^\lambda \equiv \phi_\lambda({\ku},{\kv}) C_{{\ks},{\kt}}^\lambda \pmod{A_{\succ\lambda}}.$$
The scalar $\phi_\lambda({\ku},{\kv}) \in K$ depends exclusively on the inner indices $\ku$ and $\kv$, and is independent of the outer indices $\ks$ and $\kt$.

The \emph{canonical bilinear form} associated with $\lambda$ is defined as the unique $K$-bilinear map $\phi_\lambda : W(\lambda) \times W(\lambda) \to K$ given on the basis elements of the cell module $W(\lambda)$ by:
$$\phi_\lambda(C_{\ku}, C_{\kv}) = \phi_\lambda({\ku},{\kv}).$$

\end{definition}

\begin{lem}\cite[Theorem 3.4]{GL96}\label{lem:radical-cell}
Let $A$ be a cellular algebra over a field $K$ and $W(\lambda)$ be the cell module associated with $\lambda \in \Lambda$. Let $\phi_\lambda : W(\lambda) \times W(\lambda) \to K$ be the canonical associative symmetric bilinear form on $W(\lambda)$. Define the radical of $\phi_\lambda$ as
$$\mathrm{rad}(\lambda) = \{ x \in W(\lambda) \mid \phi_\lambda(x, y) = 0 \text{ for all } y \in W(\lambda) \},$$
which is an $A$-submodule of $W(\lambda)$. Let $\Lambda_0 = \{ \lambda \in \Lambda \mid \phi_\lambda \neq 0 \}$ and define the quotient module $L(\lambda) = W(\lambda) / \mathrm{rad}(\lambda)$. Then:
\begin{enumerate}
    \item[(i)] For each $\lambda \in \Lambda_0$, $L(\lambda)$ is an absolutely irreducible $A$-module.
    \item[(ii)] The set $\{ L(\lambda) \mid \lambda \in \Lambda_0 \}$ constitutes a complete set of pairwise non-isomorphic simple $A$-modules.
\end{enumerate}
\end{lem}

\subsection{Cell modules}

{\color{black} In the following, we shall construct the cell modules for $\mHK$. First, the cell datum for the Hecke algebra $\mathcal{H}_d(q)$ for each $d\geq 0$, we choose the Murphy cellular basis. More precisely, for each $0\leq k\leq n$, 
the Hecke subalgebra $\mathcal{H}'_{n-k}$ admits a cell datum $(\Lambda_k,M_k,\{m^{\lambda}_{\ks,\kt}\},\sigma_k)$, where $\Lambda_k=\mathcal{P}(n-k)$ is the set of all partitions of $n-k$ and $M_k(\lambda)=\text{Std}(\lambda)$ consists of all standard Young tableaux of shape $\lambda$ and $m^{\lambda}_{\ks,\kt}$ is the Murphy basis, see \cite{Ma99} for details. 
Meanwhile, the preorder $\prec_k$ in this situation is the dominance order, that is, $\lambda\prec_k\mu$ for $\lambda=(\lambda_1,\lambda_2,\ldots),\mu=(\mu_1,\mu_2,\ldots)\in \Lambda_k$ if and only if $\lambda\neq \mu, \sum^i_{j=1}\lambda_j\leq \sum^i_{j=1}\mu_j$ for each $i\geq 1$.  
Then by Theorem \ref{thm:cellularity} we obtain that the set 
$$
\{{C}^{(k,\lambda)}_{(A,\ks),(B,\kt)}=\sigma(T^{-1}_A)P_k\, m^{\lambda}_{\ks,\kt} T_{B}^{-1}\mid 0\leq k\leq n, \lambda\in\Lambda_k, (A,\ks), (B,\kt)\in \mathcal{M}(k,\lambda)\}
$$
is a cellular basis for $\mHK$.

For $(k,\lambda)\in\Lambda$, let $\cell^{(n)}_{(k,\lambda)}$ be the cell module in the sense of Definition \ref{defn:usual-cell-module} associated with the cellular datum in Theorem \ref{thm:cellularity}.  It admits a basis parametrized by $\mathcal{M}(k,\lambda)$: 
$$\{c_{A,\ks}: A\in\binnk,\ \ks\in\Std(\lambda)\},$$
and the action of $\mathscr{H}_{n,\mathbb{K}}(q)$ is determined by the coefficients in the cellular triangularity relation given in \eqref{eq:mirabolic-basis-triangularity}:
\begin{equation}\label{eq:x-action-cell}
x\cdot c_{A,\ks}=
  \sum_{B\in\binnk}\sum_{\ku\in\Std(\lambda)}r_x(B,\ku;A,\ks)c_{B,\ku}.
  \end{equation}
  In addition, write $\langle,\rangle^{(n)}_{(k,\lambda)}$ in the sense of Definition \ref{defn:canonical form} associated with the cellular datum in Theorem \ref{thm:cellularity}. 
 For each $0\leq k\leq $ and $\lambda\in\Lambda_k=\mathcal{P}(n-k)$, let $\cell_\lambda$ denote the cell module associated to $\lambda$ for the Hecke algebra $\mathcal{H}'_{n-k,\mathbb{K}}(q)$, with basis $\{c_\ks:\ks\in\Std(\lambda)\}$. 
 Write $\langle,\rangle_{\lambda}$ the corresponding canonical bilinear form on $\cell_\lambda$ in the sense of Definition \ref{defn:canonical form} associated to the cell datum for Murphy cellular basis as above. 
 
 \begin{prop}\label{prop:gram-form}
Suppose $(k,\lambda)\in\Lambda$. Then for any $(A,\ks), (B,\kt)\in \mathcal{M}(k,\lambda)$,  we  have 
\begin{equation}\label{eq:gram-diagonal}
  \langle c_{A,\ks},c_{B,\kt}\rangle^{(n)}_{(k,\lambda)}
  =\delta_{A,B}q^{\alpha_A}\langle c_\ks,c_\kt\rangle_\lambda.
\end{equation}
\end{prop}
\begin{proof}
Take $(A,\ks),(B,\kt),(C,\ku),(D,\kv)\in \mathcal{M}(k,\lambda)$.  Using \eqref{eq:C-basis},
$$
C^{(k,\lambda)}_{(A,\ks),(B,\kt)}C^{(k,\lambda)}_{(C,\ku),(D,\kv)}
  = \sigma(T_A^{-1})m^{\lambda}_{\ks,\kt} P_k T_B^{-1}\sigma(T_C^{-1})P_km^{\lambda}_{\ku,\kv}T_D^{-1}.
  $$
By Proposition \ref{prop:XBCq} the middle factor satisfies 
$$
P_kT_B^{-1}\sigma(T_C^{-1})P_k\equiv q^{\alpha_B}\delta_{B,C}P_k  \pmod {J_{k+1}}.
$$
Hence modulo $\mathscr{H}_n^{\succ(k,\lambda)}$,
\begin{equation}\label{eq:canonical-form-1}
C^{(k,\lambda)}_{(A,\ks),(B,\kt)}C^{(k,\lambda)}_{(C,\ku),(D,\kv)}
  \equiv q^{\alpha_B}\delta_{B,C}\,\sigma(T_A^{-1})P_km^{\lambda}_{\ks,\kt} m^{\lambda}_{\ku,\kv}T_D^{-1} \pmod{J_{k+1}}.
  \end{equation}
  Meanwhile in the Hecke algebra $\mathcal{H}'_{n-k,\mathbb{K}}(q)$, we have 
  \begin{equation}\label{eq:canonical-form-2}
  m^{\lambda}_{\ks,\kt} m^{\lambda}_{\ku,\kv}
  \equiv \langle c_\kt, c_\ku\rangle_\lambda m^{\lambda}_{\ks,\kv}
  \pmod{(\mathcal{H}'_{n-k,\mathbb{K}}(q))^{\succ\lambda}}.
  \end{equation}
Putting \eqref{eq:canonical-form-1} and \eqref{eq:canonical-form-2} together we obtain as in the proof of Theorem \ref{thm:cellularity}: 
$$
\begin{aligned}
C^{(k,\lambda)}_{(A,\ks),(B,\kt)}C^{(k,\lambda)}_{(C,\ku),(D,\kv)}
&\equiv q^{\alpha_B}\delta_{B,C}\langle c_\kt, c_\ku\rangle_\lambda \sigma(T_A^{-1})P_k m^{\lambda}_{\ks,\kv} T_D^{-1} \pmod{(\mathscr{H}_{n,\mathbb{K}}(q))^{\succ (k,\lambda)}}\\
&\equiv q^{\alpha_B}\delta_{B,C}\langle c_\kt, c_\ku\rangle_\lambda C^{(k,\lambda)}_{(A,\ks),(D,\kv)} \pmod{(\mathscr{H}_{n,\mathbb{K}}(q))^{\succ (k,\lambda)}}. 
\end{aligned}
$$
Then by Definition \ref{defn:canonical form}  we obtain 
\begin{equation}\label{eq:canonical-form-3}
 \langle c_{A,\ks},c_{B,\kt}\rangle^{(n)}_{(k,\lambda)}
  =\delta_{A,B}q^{\alpha_A}\langle c_\ks,c_\kt\rangle_\lambda.
  \end{equation}
This proves the proposition. 
\end{proof}

For each $(k,\lambda)\in\Lambda$, $D^{(n)}_{(k,\lambda)}:=\Delta^{(n)}_{(k,\lambda)}/\text{rad}\langle,\rangle^{(n)}_{(k,\lambda)}$. 
\begin{cor}\label{cor:explicit-irred} 
For each $(k,\lambda)\in\Lambda$,  $\langle,\rangle^{(n)}_{(k,\lambda)}\neq 0$ if and only if $\langle,\rangle_\lambda\neq 0$. 
Thus $D^{(n)}_{(k,\lambda)}\neq 0$ if and only if $\lambda$ is $e$-restricted or equivalently the set $\{D^{(n)}_{(k,\lambda)}\mid 0\leq k\leq n, \lambda\in\mathcal{RP}_e(n-k)\}$ 
is a complete set of pairwise non-isomorphic irreducible $\mHK$-modules. In particular, if $e>n$, then every $\cell^{(n)}_{(k,\lambda)}=D^{(n)}_{(k,\lambda)}$ is irreducible. 

\end{cor}
For $0\leq k\leq n$ and $\lambda\in\mathcal{P}_{n-k}$, let $G^{(n)}_{(k,\lambda)}$ be the Gram matrix associated to the canonical form $\langle,\rangle^{(n)}_{(k,\lambda)}$ with respect to the basis $c_{(A,\ks)}$ with $(A,\ks)\in \mathcal{M}(k,\lambda)$. 
Meanwhile, let $G_{\lambda}$ be the Gram matrix associated to the canonical form $\langle,\rangle_\lambda$. By Proposition \ref{prop:gram-form} we have the following.   

\begin{cor} For each $(k,\lambda)\in\Lambda$, 
\begin{equation}\label{eq:det-factor}
  \det G^{(n)}_{(k,\lambda)}
  =q^{f^\lambda\sum_{A\in\binnk}\alpha_A}(\det G_\lambda)^N, 
\end{equation}
where $f^\lambda$ is the number of standard Young tableaux of shape $\lambda$ and $N=\binom{n}{k}$. 
\end{cor}
 
\begin{rmk}
Though we already know the classification of irreducible modules in Corollary \ref{cor:semisimplicity}, we choose to include Corollary \ref{cor:explicit-irred} as it provides an construction of irreducible modules via an explicit calculation using the cellular basis. 
\end{rmk}

%%%%%%%%%%%%%%%

\subsection{ Identification with  Siegel's irreducible modules}
Throughout this subsection $\mathbb K=\C(q)$. 

\begin{lem}\label{lem:siegel}\cite[Proposition 3.25]{Si93}
For every $0\le r\le n$ and every partition $\lambda\vdash r$, Siegel constructs an irreducible $\mHK$-module $N^{(n)}_{\lambda,r}$, and the modules
$$\{N^{(n)}_{\lambda,r}:0\le r\le n,
     \lambda\vdash r\}$$
form a complete set of pairwise nonisomorphic irreducible modules.
\end{lem}

%\begin{rmk}
%	In \cite[Proposition 3.8]{Ro14},  Rosso first realized  simple $\mHK$-module $N^{(n)}_{\lambda,r}$ as a $\mathcal H_{n,\mathbb K}(q)$-module, he then identifies this module with the cyclotomic quotient simple module in \cite[Proposition 4.15]{Ro14}.
%\end{rmk}

\begin{lem}\label{lem:cell-projections}
Let $0\le k\le n$ and $\lambda\in\mathcal{P}(n-k)$. Then  the following holds.
\begin{enumerate}
\item If $j>k$, then $P_j\cell^{(n)}_{(k,\lambda)}=0$.  In particular, if $k<n$, then
$P_{k+1}\cell^{(n)}_{(k,\lambda)}=0$. 
\item $P_k\cell^{(n)}_{(k,\lambda)}\neq 0$. Moreover $P_k\cell^{(n)}_{(k,\lambda)}\cong \cell_\lambda$ as a $\Hta$-module,
\item If $j<k$, then $P_j\cell^{(n)}_{(k,\lambda)}\neq 0$. 
\end{enumerate}
\end{lem}

\begin{proof}
    If $j>k$, then $P_j\in J_j\subseteq J_{k+1}$.  Then by \eqref{eq:h-definition} we have $h^{(k)}_{C,A}(P_j)=0$ for any $A,C\in\mathscr{F}^n_k$ and hence by \eqref{eq:x-action-cell} and \eqref{eq:mirabolic-basis-triangularity} we obtain $P_jc_{A,\ks}=0$ for $(A,\ks)\in\mathcal{M}(k,\lambda)$. Hence  $P_j\cell^{(n)}_{(k,\lambda)}=0$. This proves (1). Similarly, one can use \eqref{eq:h-definition}, \eqref{eq:x-action-cell} and \eqref{eq:mirabolic-basis-triangularity} to get 
\begin{equation}\label{eq:Pk-action}
  P_k\, c_{A,\ks}=
  \begin{cases}
    c_{A,\ks},& A=[k],\\
    0,& A\ne [k], 
  \end{cases}
\end{equation}
where we recall $[k]=\{1,2,\ldots,k\}$. Therefore $P_k\cell^{(n)}_{(k,\lambda)}=\mathbb{K}$-span $\{c_{[k],\ks}\mid \ks\in\text{Std}(\lambda)\}$. In addition, for each $x\in \mathcal{H}'_{n-k,\mathbb{K}}(q)$, we have 
$x\, c_{[k],\ks}=\sum_{\kt}\mathtt{r}_{\ks,\kt}c_{[k],\kt} $ if $x\, c_{s}=\sum_{\kt}\mathtt{r}_{\ks,\kt}c_\kt \in \cell_\lambda$ with $\mathtt{r}_{\ks,\kt}\in\mathbb{K}$
again by \eqref{eq:h-definition}, \eqref{eq:x-action-cell} and \eqref{eq:mirabolic-basis-triangularity}. This means $P_k\cell^{(n)}_{(k,\lambda)}\cong \cell_\lambda$ as a $\Hta$-module.  This proves (2). 

Suppose $j<k$.   Since $P_jP_k=P_k$ for $j<k$, we have by \eqref{eq:Pk-action}
$$P_j\, c_{[k],\ks}=P_j\,(P_k\,(c_{[k],\ks}))=P_k\, c_{[k],\ks}=c_{[k],\ks}\ne0.$$
Thus (3) is proved. 
\end{proof}

\begin{lem}\label{lem:siegel-projection}
Let $0\le k,l\le n$, $\mu\in\mathcal{P}(l)$.  For the Siegel module $N^{(n)}_{\mu,l}$, the following statements hold.
\begin{enumerate}
\item If $l>n-k$, then $P_kN^{(n)}_{\mu,l}=0$. In particular, $P_{k+1}N^{(n)}_{\mu,n-k}=0$ when $0\leq k<n$.
\item If $l=n-k$, then
$$P_kN^{(n)}_{\mu,n-k}\cong\cell_\mu$$
as a module for $\Hta$, 
\item If $l<n-k$, then $P_{k+1}N^{(n)}_{\mu,
l}\ne0$.  
\end{enumerate}
Thus, $P_nN^{(n)}_{\mu,
l}\ne0$ if and only if $l=0$.
\end{lem}

\begin{proof}
Let $V_m$ be the natural representation of the quantum group $U_q(\mathfrak{gl}_m)$ and $V_0$ be the trivial module. Denote by $L(\mu)$ the the irreducible quantum group $U_q(\mathfrak{gl}_m)$-module of highest weight $\mu$ for each partition $\mu$ with 
$\ell(\mu)\leq m$. Via Wan's mirabolic Schur--Weyl duality~\cite[Theorem~3.2]{Wan26},  we have $N_{\mu,l}^{(n)}\cong\Hom_{U_q(\mathfrak{gl}_m)}\bigl(L(\mu),(V_m\oplus V_0)^{\otimes n}\bigr)$ and moreover the action of mirabolic Hecke algebra $\mHK$ on the tensor space $(V_m\oplus V_0)^{\otimes n}$ satisfying that  the idempotent $P_k$ projects the first $k$ tensor factors onto the trivial summand. Thus $P_k(V_m\oplus V_0)^{\otimes n}\cong (V_m\oplus V_0)^{\otimes n-k}$ and hence we obtain 
\begin{equation}\label{eq:Pk-Siegel}
P_kN_{\mu,l}^{(n)}\cong \Hom_{U_q(\mathfrak{gl}_m)}\bigl(L(\mu), (V_m\oplus V_0)^{\otimes n-k}\bigr)
\end{equation}
as vector spaces. Observe that If $l>n-k$ then $(V_m\oplus V_0)^{\otimes n-k}$ 
 cannot produce a highest weight of $\mu$ with $\mu\in\mathcal{P}(l)$, so $P_kN_{\mu,l}^{(n)}=0.$ This proves (1). 

 If $l=n-k$, then \eqref{eq:Pk-Siegel} we have 
    \begin{equation}\label{eq:Pk-Siegel-1}
    P_kN_{\mu,n-k}^{(n)}\cong \Hom_{U_q(\mathfrak{gl}_m)}\bigl(L(\mu),V_m^{\otimes n-k}\bigr).
    \end{equation}
    Meanwhile, the subalgebra  $\mathcal{H}'_{n-k,\mathbb{K}}(q)$commutes with $P_k$. Thus  the action of $\mathcal{H}'_{n-k,\mathbb{K}}(q)$ on $P_kN_{\mu,n-k}^{(n)}$ coincides with its action of on the tensor space $V_m^{\otimes n-k}$ defined in \cite[(3.4)]{Wan26}. 
    So the actions of $U_q(\mathfrak{gl})_m$ and  $\mathcal{H}'_{n-k,\mathbb{K}}(q)$ on $V_m^{\otimes n-k}$ satisfies the double centralizer property. Then by a comparison with the actions defined in \cite{Ji86} one can obtain that  $\Hom_{U_q(\mathfrak{gl}_m)}\bigl(L(\mu),V_m^{\otimes n-k}\bigr)\cong \cell_\mu$ by \cite[Section 3.2]{Ma99}. Thus using \eqref{eq:Pk-Siegel-1} we obtain 
    $$
    P_kN_{\mu,r}^{(n)}\cong \Hom_{U_q(\mathfrak{gl}_m)}\bigl(L(\mu),V^{\otimes r}\bigr)\cong \cell_\mu 
    $$
 as $\mathcal{H}'_{n-k,\mathbb{K}}(q)$-modules. This proves (2). 

By \eqref{eq:Pk-Siegel}, we obtain that $P_aN_{\mu,b}^{(n)}\neq 0$ whenever $a+b\leq n$ for each $0\leq a,b\leq n$. So (3) follows. Then the last statement is obvious. 
  \end{proof}

\begin{rmk}
	The above proof of Lemma \ref{lem:siegel-projection} applies properties of Schur–Weyl duality.  In fact, one can use Rosso's explicit construction of simple $\mHK$-modules $N^{(n)}_{\mu,l}$ in \cite[Proposition 3.8 and 4.15]{Ro14} to compute $P_kN^{(n)}_{\mu,l}$ to give an alternative proof of Lemma \ref{lem:siegel-projection}.
\end{rmk}

\begin{thm}\label{thm:recover-siegel}
For every $0\le k\le n$ and  $\lambda\in\mathcal{P}(n-k)$,
\begin{equation}\label{eq:recover-siegel}
  \cell^{(n)}_{(k,\lambda)}
  \cong N^{(n)}_{\lambda,n-k}.
\end{equation}
\end{thm}

\begin{proof}
By Corollary \ref{cor:explicit-irred},  $\cell^{(n)}_{(k,\lambda)}$ is simple. Then by Lemma~\ref{lem:siegel}, $\cell^{(n)}_{(k,\lambda)}\cong N^{(n)}_{\mu,l}$ for a unique pair $(\mu,l)$. 
We claim $l=n-k$. Otherwise, if $l>n-k$, then $P_k \cell^{(n)}_{(k,\lambda)}\neq 0$ while $P_{k}N^{(n)}_{\mu,l}=0$ by Lemma \ref{lem:siegel-projection} and Lemma
\ref {lem:cell-projections}, a contradiction. If $l\leq n-k-1$, then $k\leq n-1$ and $P_{k+1}\cell^{(n)}_{(k,\lambda)}= 0$ while $P_{k+1}N^{(n)}_{\mu,l}\neq 0$ by Lemma \ref{lem:siegel-projection} and Lemma 
\ref {lem:cell-projections}, a contradiction. Thus $l=n-k$.  This means $\cell^{(n)}_{(k,\lambda)}\cong N^{(n)}_{\mu,n-k}$. Hence $P_k\cell^{(n)}_{(k,\lambda)}\cong P_kN^{(n)}_{\mu,n-k}$, which implies  $\cell_\lambda\cong\cell_\mu$ as $\mathcal{H}'_{n-k,\mathbb{K}}(q)$-modules
by Lemma \ref{lem:siegel-projection} and Lemma  
\ref {lem:cell-projections}, Hence $\mu=\lambda$. This proves the theorem. 

\end{proof}

%%%%%%%%%
%%%%%%%%%

\end{document}